\documentclass{article}

\usepackage{arxiv}

\usepackage{lineno,hyperref}
\modulolinenumbers[50]
\usepackage{adjustbox}
\usepackage{algorithm}
\usepackage{algorithmic}
\usepackage{amsmath,amsfonts,amssymb,array}
\usepackage{bbm}
\usepackage{booktabs}
\usepackage{cancel}
\usepackage{caption}
\usepackage{color,colortbl}
\usepackage{epsfig}
\usepackage{epstopdf}
\usepackage{float}
\usepackage{graphicx}
\usepackage{hyperref}
\usepackage{lipsum}
\usepackage{lineno}
\usepackage{longtable}
\usepackage{mathtools}
\usepackage{multicol}
\usepackage{multirow}
\usepackage{pdflscape}
\usepackage{relsize}
\usepackage{rotating}
\usepackage{soul}
\usepackage{subfig}
\usepackage[dvipsnames,table]{xcolor}
\usepackage{tabularx}
\usepackage{tabulary}
\usepackage{textcomp}

\newcommand{\norm}[1]{\left|\left|{#1}\right|\right|}
\newcommand{\I}{\mathcal{I}}
\newcommand{\T}{\mathcal{T}}

\newcommand*{\QEDA}{\hfill\ensuremath{\blacksquare}}%

\newtheorem{lemma}{Lemma}
\newtheorem{proposition}{Proposition}

\newtheorem{proof}{Proof}

\title{3-D Dynamic UAV Base Station Location Problem}

\author{
  Cihan~Tugrul~Cicek \\
  Department of Industrial Engineering \\
  Atilim University \\
  Ankara, Turkey \\
  \texttt{cihan.cicek@atilim.edu.tr} \\
  \And
  Zuo-Jun~Max~Shen \\
  Department of Industrial Engineering and Operations Research \\
  University of California, Berkeley\\
  Berkeley, CA, USA \\
  \texttt{maxshen@berkeley.edu} \\
  \And
  Hakan~Gultekin \\
  Department of Mechanical and Industrial Engineering \\
  Sultan Qaboos University\\
  Muscat, Oman \\
  \texttt{hgultekin@squ.edu.om} \\
 \And
  Bulent~Tavli \\
  Department of Electrical and Electronics Engineering\\
  TOBB University of Economics and Technology\\
  Ankara, Turkey \\
  \texttt{btavli@etu.edu.tr} \\
}

\begin{document}
\maketitle

\begin{abstract}
We address a dynamic covering location problem of an Unmanned Aerial Vehicle Base Station (UAV-BS), where the location sequence of a single UAV-BS in a wireless communication network is determined to satisfy data demand arising from ground users. This problem is especially relevant in the context of smart grid and disaster relief. The vertical movement ability of the UAV-BS and non-convex covering functions in wireless communication restrict utilizing classical planar covering location approaches. Therefore, we develop new formulations to this emerging problem for a finite time horizon to maximize the total coverage. In particular, we develop a mixed-integer non-linear programming formulation which is non-convex in nature, and propose a Lagrangean Decomposition Algorithm (LDA) to solve this formulation. Due to high complexity of the problem, the LDA is still unable to find good local solutions to large-scale problems. Therefore, we develop a Continuum Approximation (CA) model and show that CA would be a promising approach in terms of both computational time and solution accuracy. Our numerical study also shows that the CA model can be a remedy to build efficient initial solutions for exact solution algorithms.
\end{abstract}


\section{Introduction}
Unmanned aerial vehicles (UAVs) as aerial base stations are a promising technology to provide wireless services to ground users. For example, UAV Base Stations (UAV-BSs) can provide emergency communication services to rescue teams after a disaster or can help improving wireless capacity in congested networks \cite{Bor-Yaliniz2016}. Among all other opportunities, rapid deployment and mobility advantages of UAV-BSs enforce players in the telecommunication sector to involve aerial components (e.g., drones or balloons) besides existing terrestrial components (e.g., macro-cell base stations) in their operations. In this way, the quality of service and the quality of experience can be improved. quality of service typically includes performance metrics on how service providers monitor the service provision such as the cost of service or coverage. quality of experience, on the other hand, includes metrics on how users perceive the services such as reliability or price. \cite{Cao2018}.

A key challenge to involve UAV-BSs in the design of wireless networks is to determine UAV-BS locations since the demand is instantaneous and highly dynamic in a territory over which the UAV-BS is planned to serve. Recent works have studied the single and multiple UAV-BS location for the static networks, where the demand and location that demand is originated from are fixed and known. However, compared with static networks, the UAV-BS location problem for dynamic networks is more challenging when the demand and locations are random. Although there have been some attempts to solve this challenging dynamic problem, much uncertainty still exists about how to incorporate different behaviors of ground users into the problem such as changing locations and different service level requests at different time epochs. 

To solve this problem, we introduce a new 3-D Maximal Covering Location Problem (3MCLP) to model a wireless communication network, in which a single UAV-BS is used to maximize coverage in a finite service area by offering wireless services to users who have dynamic demand within a finite time horizon. Although the coverage functions have typically been assumed to be convex and dependent solely on the distance in the classical covering location problems, the coverage function in wireless networks considers both the distance and the angle between a particular user and the UAV-BS. Therefore, classical covering location approaches cannot help to solve this new problem. 

In this study, we deliberately address the 3MCLP of the UAV-BS in a dynamic environment where both the demand and locations of users change in time. Our objective is to determine the optimal location sequence of the UAV-BS for a limited time horizon to maximize the overall coverage while satisfying user demand. In particular, we first develop a non-convex Mixed Integer Non-Linear Programming (MINLP) formulation for this dynamic problem and propose two algorithms to solve this problem. First, we develop a Lagrangean Decomposition Algorithm (LDA), where we can decompose the original problem into relatively easier to solve smaller problems after applying Lagrangean Relaxation (LR) technique with additional convex approximations of some non-convex functions. Second, we develop a Continuum Approximation (CA) model to reformulate the problem and propose a CA algorithm to solve this new model. We test the two algorithms through a computational study with synthetic data.

The main contributions of our study can be summarized as follows:
\begin{enumerate}
    \item A new 3-D covering location problem is introduced, where the vertical dimension and the temporal change in demand is involved.
    \item A non-convex coverage function is proposed, where the coverage level does not solely depend on the distance to the UAV-BS, but is a function of both the distance and the angle between the users and the UAV-BS.
    \item A new MINLP formulation is developed, and its characteristics are analyzed.
    \item Two new algorithms are developed as a remedy to overcome the computational burden of the discrete formulation, which have improved the computational efficiency.
\end{enumerate}

The rest of the paper is organized as follows: A comprehensive literature review is given in Section~\ref{sec:literature}. The system model is given in Section~\ref{sec:model}, and a discrete programming formulation of the problem is developed in Section~\ref{sec:discrete} together with the LDA to solve this problem. Section~\ref{sec:ca} provides the details of the CA approach. We present computational results of the proposed solution algorithms in Section~\ref{sec:results}, and conclude the study in Section~\ref{sec:conclusion}.

\section{Literature review}\label{sec:literature}
The promising research including UAV-BSs (such as enhancing network capacity, improving Quality-of-Experience, extending coverage) have been attracting significant interest despite the relatively new appearance of the topic. Especially, the location optimization problems have attracted significant interest since they have a crucial impact on wireless network performance \cite{Cicek2019a}. In this section, we present a high-level overview of the UAV-BS literature with a special focus on location problems together with its reflection in the Operations Research literature and present the differences of our study from the existing literature.

UAV-assisted wireless networks are started to be investigated with the launch of emerging technologies in the design of heterogeneous telecommunication networks where both terrestrial and aerial components can be jointly utilized \cite{Cao2018}. Among all others, one of the challenges in such networks is to optimally determine the UAV-BS locations not only to improve the communication channel performance between the UAV-BSs and users, but also to maintain reliable backhaul connection with the satellite or the terrestrial networks \cite{Li2019}. \cite{Mozaffari2016}, \cite{Merwaday2016}, \cite{Sharma2016}, \cite{Lyu2017}, \cite{Alzenad2017}, and \cite{Bor-Yaliniz2019} have studied static uncapacitated UAV-BS location problem with different objectives such as coverage, spectral efficiency, i.e., the unit data transmission rate per unit bandwidth, or latency, i.e., the average delay time observed when responding a data request from a user. 

The literature on UAV-enabled communications is rich and encompasses many aspects of communications~\cite{Zeng2016,Zeng2019a}. In~\cite{Xu2018}, trajectory optimization of a UAV which provides wireless energy transfer to a set of energy receivers is investigated. In~\cite{Zhan2018}, the use of UAVs to collect data from a Wireless Sensor Network (WSN) is investigated where the trajectory of the UAV and the wake-up schedule of the sensor nodes are optimized. In~\cite{Zhang2019a}, cellular enabled UAV communications paradigm is investigated. The objective is to minimize the mission completion of a UAV, which is served by terrestrial BSs, by optimizing its trajectory. In~\cite{Zhang2019b}, physical layer security aspects of UAV-enabled communications networks are addressed to maximize the average secrecy rates of the UAV-to-ground and ground-to-UAV transmissions by joint optimization of the UAV trajectory and transmit power. In~\cite{Han2020}, an approach for high reliability and low latency communication of UAV swarms is proposed which exploits both cellular communications and device-to-device communications.

In multiple UAV systems, UAVs need to communicate among themselves and such communications can be in any direction of the 3-D space. As such, many antenna positions and orientations are possible. In fact, it is shown in~\cite{Badi2019} that UAVs themselves can act as local scatterers which results in elevated channel depolarization and decreased cross-polarization discrimination (XPD). Furthermore, experimental analysis reveals that relative direction of UAVs has a significant impact on XPD values, 3-D link performance is critically affected by the elevation angle, and spectral efficiency can be more than doubled by cross-polarized antenna configurations.

To improve the overall service quality and cover more realistic cases, \cite{Kalantari2017a}, \cite{Kalantari2017b}, \cite{Kalantari2017c}, \cite{Shi2018}, and \cite{Cicek2019b} revisited the problem by incorporating the capacity of the network, e.g. maximum data rate that can be offered to the users, into the problem setup and developed new problem structures and solution approaches. However, none of these studies consider the temporal change in the parameters and mainly focused on a specific time epoch or a snapshot of the network.

While there is an extensive literature investigating UAV-BS location models in the static setup, there has been surprisingly little work investigating how dynamic problem setup where both UAV-BSs and users are allowed to move within the planning horizon can change the network performance. \cite{Fotouhi2016} solve a single UAV-BS location problem to improve the spectral efficiency for a finite time horizon with three different heuristic algorithms, where the UAV-BS altitude and velocity are assumed to be fixed. \cite{Mozaffari2017} study the energy aspect of the UAV-BSs. The authors propose three heuristic algorithms to find the locations of multiple UAV-BSs to minimize the total transmit power for serving multiple Internet-of-Things devices located on the ground. Again, the UAV altitudes are assumed to be fixed. 

Propulsion energy dissipation of UAVs as a function of speed has been investigated in detail in the literature. In fact, it is shown that minimum energy dissipation is achieved at a certain non-zero speed (i.e., not hovering)~\cite{Phung2013,DiFranco2015,Karydis2017,Zeng2017,Zeng2019b,Ebrahimi2020}. In~\cite{Zeng2017,Zeng2019b}, UAV trajectory optimization frameworks by considering both communications objectives and energy dissipation minimization (by utilizing the aforementioned relationship between propulsion energy dissipation and speed) are proposed and significant performance improvements are reported.

\cite{Liu2018} study the single multiple UAV-BS location problem to design an energy-efficient wireless network, in which the maximum service area of the UAV-BSs are assumed to be known and fixed, and propose a deep reinforcement learning algorithm. Recently, \cite{WangZ2019a} study an adaptive UAV location problem where the altitude of the UAV-BS is assumed to be fixed and the objective is to maximize the average throughput provided to users. An adaptive scheme is proposed in which ground users are divided into two or four identical sectors, and a single sector is selected to locate the UAV-BS at each time epoch by maximizing the coverage probability. However, none of these studies consider: (i) varying user demand and (ii) moving ability of the UAV-BSs in the vertical dimension.  

There exist more studies related to the routing decisions of single or multiple UAVs to optimize different objectives in a dynamic setup such as data package delivery. However, such studies require a different understanding of the networks with a focus on the hovering ability of the UAVs like the velocity and/or the order of which users or ground terminals are served. Therefore, those studies are typically treated as routing problems, which are fundamentally different from covering problems, thus, these studies are out-of-scope of our study, and are not included in our literature review. A broad review of the technologies used in UAV-assisted networks and various applications of UAV-BSs in telecommunication can be found in the reviews of \cite{Hayat2016}, \cite{Cao2018}, and \cite{Li2019}. 

Although the research on UAV-BS location is fairly recent, covering location problems have been studied for decades in Operations Research community. After the seminal work of \cite{Church1974} in which the locations of a fixed number of public facilities are determined to cover as much population as possible, MCLP has become an essential component for a wide range of applications such as disaster management, fire protection, public services, and telecommunications \cite{Brotcorne2003, Gentili2012, Tu2016, Cordeau2019, Chauhan2019}. Readers are referred to \cite{Farahani2012} and \cite{Murray2016} for general MCLP problems and to \cite{Berman2010} for extensions of MCLP. 
Since the goal of almost all of the existing studies on MCLP are mainly based on determining the locations of buildings, such as warehouses or plants, the utilized models are typically limited to $\mathbb{R}^2$. However, it is envisioned that UAVs will play a crucial role in transforming several industries to technology-driven operations \cite{Gupta2016} and it would not be surprising that widespread adoption of UAVs will occur progressively in many more applications and industries to improve efficiency and operability. 

Aligned with this prediction, some studies have considered using UAVs as an emerging tool to improve the system performance in different industries. Humanitarian logistics has been one of the highly popular areas where UAVs are used as complementary tools to extend and leverage disaster relief operations \cite{Erdelj2017,Malandrino2019}. Last-mile delivery and emergency response planning have been some other areas where UAVs are utilized due to their cost advantage and mobility \cite{Murray2015, Haidari2016, Pulver2016, Kim2017}. A recent survey on different civil applications of UAVs can be found in \cite{Alena2018}.

A significant drawback of the current studies related to UAV location is to ignore the flexibility of the UAVs to dynamically change their position both in horizontal and vertical directions. However, integrating this ability into the problem formulation increases the complexity substantially. Therefore, two common approaches are adopted: (i) fixing the UAV altitudes or (ii) discarding the decision variables related to the vertical dimension by some simplifying assumptions, which yield sub-optimal decisions. 

Apart from simplifying a challenging task by introducing different relaxations or assumptions on the vertical dimension, in most of the studies, mathematical formulations have adopted discrete programming techniques. Thus, even moderately sized problems easily become intractable. In this study, we combine the flexibility of positioning the UAV-BSs in the vertical dimension and a more realistic non-convex coverage function to develop a mathematical programming formulation and propose two solution algorithms to mitigate sub-optimal decisions.

\section{System Model}\label{sec:model}

We consider a UAV-assisted wireless communication network where a single UAV-BS is used to provide wireless services to multiple users within a finite horizon. It is assumed that locations and demand of users are known a priori and the UAV-BS is fed with this information to design its trajectory before-hand. Such an approach is adopted in different studies with different objectives such as maximizing average information rate \cite{Jiang2019}, minimizing energy consumption \cite{WangZ2019b}, and maximizing secrecy \cite{Zhang2019b}.

The UAV-BS is assumed to have an infinite capacity due to dedicated backhaul links to ground base stations. That is, the demand can be infinitely satisfied without regarding the capacity of the UAV-BS, since any capacity requirement can be instantly fulfilled by allocating more bandwidth from the ground base stations \cite{Alzenad2017,Ghanavi2018}. Since backhaul requires wireless links, interference between UAV-BS-to-user (fronthaul) and backhaul links should be considered. Similar to interference avoiding methods studied in the literature \cite{Qiu2020,Mozaffari2015}, we assume that different frequency bands are employed in fronthaul and backhaul links. Moreover, the time-division duplexing (TDD) is assumed, where equally allocated time slots are adopted. In such a scheme, the UAV-BS is assumed to receive a signal from the users in the first time slot and forward signals to ground base stations in the next slot.

Relaxing the backhaul capacity limitation and interference would decrease the complexity of the problem, however, thoroughly analyzing this simpler setup will provide powerful insights for designing more complex structures. Thus we keep the capacitated network design as a future research, and exclude ground base stations from the model.

\subsection{Proposed network structure}

Let $S \subseteq \mathbb{R}^2$ and $Q \subseteq \mathbb{R}^3$ be the convex and bounded regions within which demand arises and the UAV-BS flies, respectively. We use $\I=\{1,\ldots ,n\}$ and $\T=\{1,\ldots ,T\}$ to denote demand nodes and time intervals, respectively. $n$ users are assumed to move inside $S$ during the entire time horizon as well as no additional users are allowed to enter. Similarly, the UAV-BS is assumed to serve inside $Q$ for all intervals. 

The UAV-BS is assumed to serve for a limited time horizon due to battery concerns and can move at specific time epochs to respond to the changing environment. Note that the UAV-BS is assumed to be capable of performing successful flying operations independent of the duration and number of intervals. However, we incorporate a penalty factor to our model to prevent aggressive movements of UAV, which is explained in the sequel.

The time epochs are assumed to be fixed, so that the entire time horizon can be divided into fixed intervals. Each user is associated with a non-negative weight, $w_i\in[0,1]$, and assumed to be covered in a specific time interval if and only if interval-specific Maximum Signal Loss Threshold (MSLT) of the user, $d_{it} \in \mathbb{R}$, is not exceeded. The details of signal loss is given in the next section. This loss is similar to the spatial distance constraint of the classical MCLP, where a facility is assumed to cover a customer unless the distance is greater than a maximum threshold value.

We assume that the coverage level of a user depends solely on the signal quality provided by the UAV-BS. This quality may depend on several other factors such as transmit power at the transmitter and the obstructions on or near the signal path. However, we assume that the signal power at the UAV-BS is fixed and all other factors can be ignored. In particular, the signal quality can be measured by a loss function, $L:Q\times S \rightarrow \mathbb{R}$, that depends on the locations of the users and UAV-BS (see Section~\ref{sec:slfunction} for a detailed explanation). 

Let $y_{it} \in S$ be the location of user $i \in \I$ in interval $t \in \T$. The UAV-BS will be located at $x_t \in Q$ in interval $t\in\T$ and cover a set of users in its covering area $C_t(x_t) \subseteq \I$. A user is fully covered in a particular interval if loss level in that interval is below a certain value and the coverage level gradually decreases until loss level exceeds MSLT of the corresponding user. Thus, $C_t$ can be defined as $C_t(x_t) \coloneqq \{i: L(x_t,y_{it}) \leq d_{it}\}$. Throughout this paper, unless otherwise specified, we use boldface capital letters to denote matrices and lower-case letters to denote vectors consisting of scalar parameters or variables denoted by the same letter, e.g., $\mathbf{w}$ is the vector whose components are $w_i$ for $i \in \I$, and $\mathbf{Y}$ is the matrix whose components are $y_{it}$ for $i \in \I$ and $t \in \T$. 

An illustration of the described wireless network is given in Fig.~\ref{fig1}. Users may move within $S$ between time intervals. The discs depict the covering area in an interval with cross symbols depicting the projection of the UAV-BS onto the ground. We assume that user MSLTs are independent of each other in each interval. A user is covered in a specific interval if the signal loss is below the corresponding MSLT value of that interval. Hence, users who are covered in an interval may not be covered in the next interval or vice-versa. Users with blue and red color in Fig.~\ref{fig1} depict the users who are covered and not covered in an interval, respectively. The change in the coverage in different time intervals can be observed in this figure. Non-coverage can arise in two cases: (i) relocation of a user or the UAV-BS can degrade the signal and the new signal quality may not satisfy user demand anymore or (ii) user demand (MSLT value) can change and the loss level may exceed the new value.

\begin{figure}[!t]
 \centering
 \vspace{-0.1in}
    \begin{tabular}{c}
    \subfloat[$t=t_1$.]
    {\resizebox{0.6\textwidth}{!}{\includegraphics[scale=0.34]{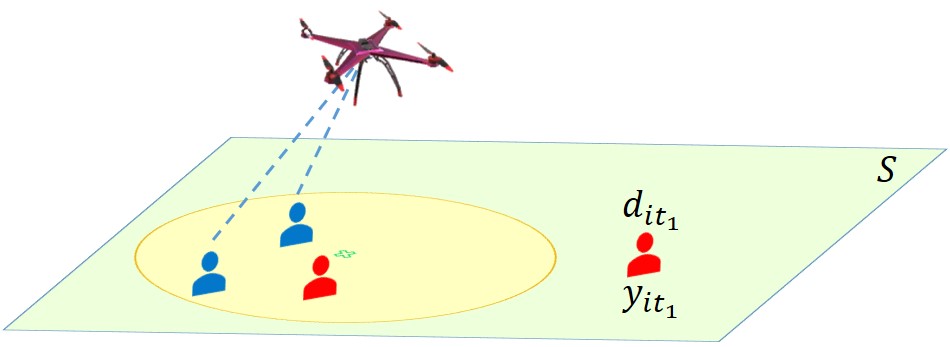}\label{fig1a}}} \\
    \subfloat[$t=t_2$.]
    {\resizebox{0.6\textwidth}{!}{\includegraphics[scale=0.34]{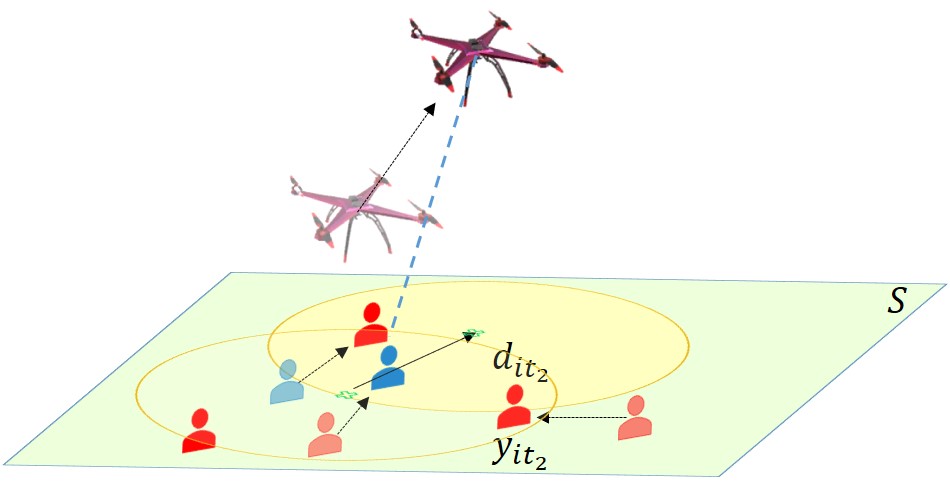}\label{fig1b}}}
    \end{tabular}
    \caption{Illustration of network for different time intervals.}
    \label{fig1}
\end{figure}

We made several simplifying assumptions in our model. However, arguably, the most important communication theoretic aspect of UAV-BSs is the path loss function itself when compared to many other applications of the communication theory. We adopt the path loss model for UAV-assisted networks proposed by \cite{Al-Hourani2014}, which has been heavily cited in the literature. Indeed, avoiding extremely complicated models by focusing on the most important aspects of UAV-BS communications is important for extracting invaluable insights which otherwise would be extremely challenging (if not impossible) to unearth.

\subsection{Service function}\label{sec:slfunction}

UAV-assisted wireless communication networks have several unique characteristics when compared to existing terrestrial communication networks. One of the most significant differences is based on the signal loss, namely Path Loss (PL), between transmitters and receivers, which are the UAV-BS and users, respectively, in our context. The higher the PL is the higher the decrease of transmitted signal power. It is not possible to successfully decode the transmitted signal if the received signal power is lower than a threshold. It means that the received power level should not be much lower than the noise level. Therefore, higher transmit power should be utilized to compensate for the higher PL levels. In the classical terrestrial networks, base stations are equipped with transmitters that can provide much higher transmission power levels to compensate for high PLs. However, mobile vehicles like UAVs have less space and capacity to handle such high power transmission equipment. Therefore, managing PL in a UAV-assisted network is more important.

There exist several PL models in the literature to define the loss level in UAV-assisted networks \cite{Khawaja2018}. We adopt the model proposed by \cite{Al-Hourani2014}, which has been the most cited model to date. In this model, users are assumed to be divided into two groups, where the first group has the probability of having Line-of-Sight (LoS) connections with low PL, while the second group does not have LoS but still can maintain a relatively poor connection with high PL.  

Let $r(x,y)=\norm{K(x-y)}$, $H(x,y)=\norm{V(x-y)}$ and $\theta(x,y)=(180 \slash \pi)\arctan \left( H(x,y) \slash r(x,y)\right)$ be the horizontal and vertical distances, and the elevation angle between a UAV-BS located at $x \in Q$ and a user at $y \in S$, respectively. Here $\norm{\cdot}$ is the Euclidean norm, $K$ and $V$ are linear transformations defined to transform 3-D location vectors to horizontal ($\mathbb{R}^2$) and vertical ($\mathbb{R}$) location vectors, respectively. Then, the signal loss between the UAV-BS at $x\in Q$ and a user at $y\in S$ is defined as the following non-convex function,
\begin{equation}
    L(x,y) = F + 10\eta\log_{10}\left( \norm{x-y}\right) +\frac{B}{1+\alpha e^{-\beta(\theta(x,y)-\alpha)}},
\label{eqn:pl}
\end{equation}

\noindent where $F=10\eta\log_{10}\left(\frac{4\pi f}{c}\right)+\phi_{\text{NLoS}}$ and $B=\phi_{\text{LoS}}-\phi_{\text{NLoS}}$ are constant loss parameters with channel frequency $f$ in Hz, speed of light $c$ in $\text{m}\slash\text{s}$. The parameters $\eta$, $\alpha$, $\beta$, $\phi_{\text{LoS}}$ and $\phi_{\text{NLoS}}$ depend on the environment, which can be a suburban, urban, dense urban or high-rise urban. The UAV-BSs typically have a finite service area, thus $\norm{x-y}$ and $\theta(x,y)$ terms are also finite, while there always exists a constant loss value due to $F$ in this function. Therefore, we can assume that $L(\cdot)$ is bounded (i.e., $L \in [L^-$,$L^+])$. 

Note that we do not include any term to model the effects of small-scale fading in Eq.~\ref{eqn:pl}, instead, we focus on large-scale path-loss. Thus, the path-loss function in Eq.~\ref{eqn:pl} can impact our analysis to some extent especially for NLOS channels because fading is often more severe in NLOS channels. Nevertheless, it is possible to mitigate the effects of small-scale fading in UAV-enabled communications systems (e.g., by utilizing an antenna array it is possible to harness the diversity of the wireless channels and achieve multiplexing gain)~\cite{Khoshkholgh2019}.

Our model uses the maximum signal loss threshold as the demand of a specific user. Such an approach can be easily extended to different demand functions such as throughput. When a user requests a certain data rate in an interval, this rate can be converted to a maximum signal loss value if the transmit power of the UAV-BS and allocated bandwidth amount are fixed \cite{Cicek2019b}. As we have similar assumptions in our model, we use the signal loss function as the demand function.

$L(.)$ is a monotonically increasing function in $r$ and unimodal in $H$. In other words, given the altitude of UAV-BS, the signal loss monotonically increases as horizontal distance increases, whereas given the horizontal distance, the loss monotonically increases up to a certain point as the altitude increases, then monotonically decreases \cite{Al-Hourani2014}. We utilize these properties, especially, to define the service area and develop solution algorithms for the CA model.  

\section{Discrete formulation}\label{sec:discrete}
In this section, we give an MINLP formulation of the described wireless communication network and develop an LDA to solve this formulation. UAVs are fast and relatively practical to use, however, there is an inherent limitation to their use in terms of battery life. Recall that we assume that the UAV-BS would have sufficient battery installed and capable of performing any flying operation within the planning horizon. Nevertheless, service time is assumed to remain limited, and we deliberately enforce the UAV-BS to avoid aggressive displacement as much as possible.

The battery of a UAV-BS is typically used for maintaining two operations simultaneously: (i) to carry the UAV and maintain a stable position in the air and (ii) to transmit the signal. Therefore, there exists a trade-off between extending the service time by hovering longer times and expanding the service area by allocating more power to transmit the signal. A promising way of managing this trade-off is to define a weighting cost parameter, $p \in [0,1]$, that allows the decision-maker to observe how different strategies such as penalizing the movement or relaxing it affects the coverage. We present a sensitivity analysis for this parameter in Section~\ref{sec:results}.

Let $\mathbf{X}=(x_1, \ldots , x_T) \in Q^T$ and $g(\mathbf{X}):Q^T \rightarrow \mathbb{R}$ be a function that defines the total relocation measurement of a location sequence of the UAV-BS. A typical definition for $g(\cdot)$ given an initial and terminating location, $x_0$ and $x_{T+1}$, is the total movement during the entire time horizon, i.e.,
\begin{equation}
    g(\mathbf{X}) \coloneqq \sum\limits_{t=1}^{T+1}\norm{x_t-x_{t-1}}.
\label{eqn:reloc}
\end{equation}

Note that there might be different definitions for $g(\cdot)$ such as the average velocity or average movement. Nevertheless, different measures have no effect on our formulation. Stability of the UAV-BS is regularized by $p\in[0,1]$. Greater $p$ values enforce the UAV-BS to keep its location stable during the entire time horizon, while smaller $p$ values allow the UAV-BS to move around. We assume that the power required during take-off and landing is negligible, or can be managed explicitly, thus, there is no cost at the beginning and end of the service time. Note that the problem can be solved independently for each interval when $p=0$.

We introduce $\mu:Q\times S \times \mathbb{R} \rightarrow [0,1]$ to define the coverage level between a UAV-BS located at $x\in Q$ and a user at $y\in S$ with MSLT $d>L^-$ as
\begin{equation}
\mu(x,y,d) \coloneqq \max\left\{0,\frac{d-L(x,y)}{d-L^-}\right\}. \label{eqn:coverage}   
\end{equation}

\noindent Observe that $\mu$ takes positive values only if the signal loss is below the given MSLT value, $d$, and gradually increases to 1 with decreasing loss values. Figure~\ref{fig2} depicts the change in $\mu$ with respect to two arbitrary MSLT values, $d_1$ and $d_2$.

\begin{figure}[!t]
    \centering
    \resizebox{.6\textwidth}{!}{\includegraphics{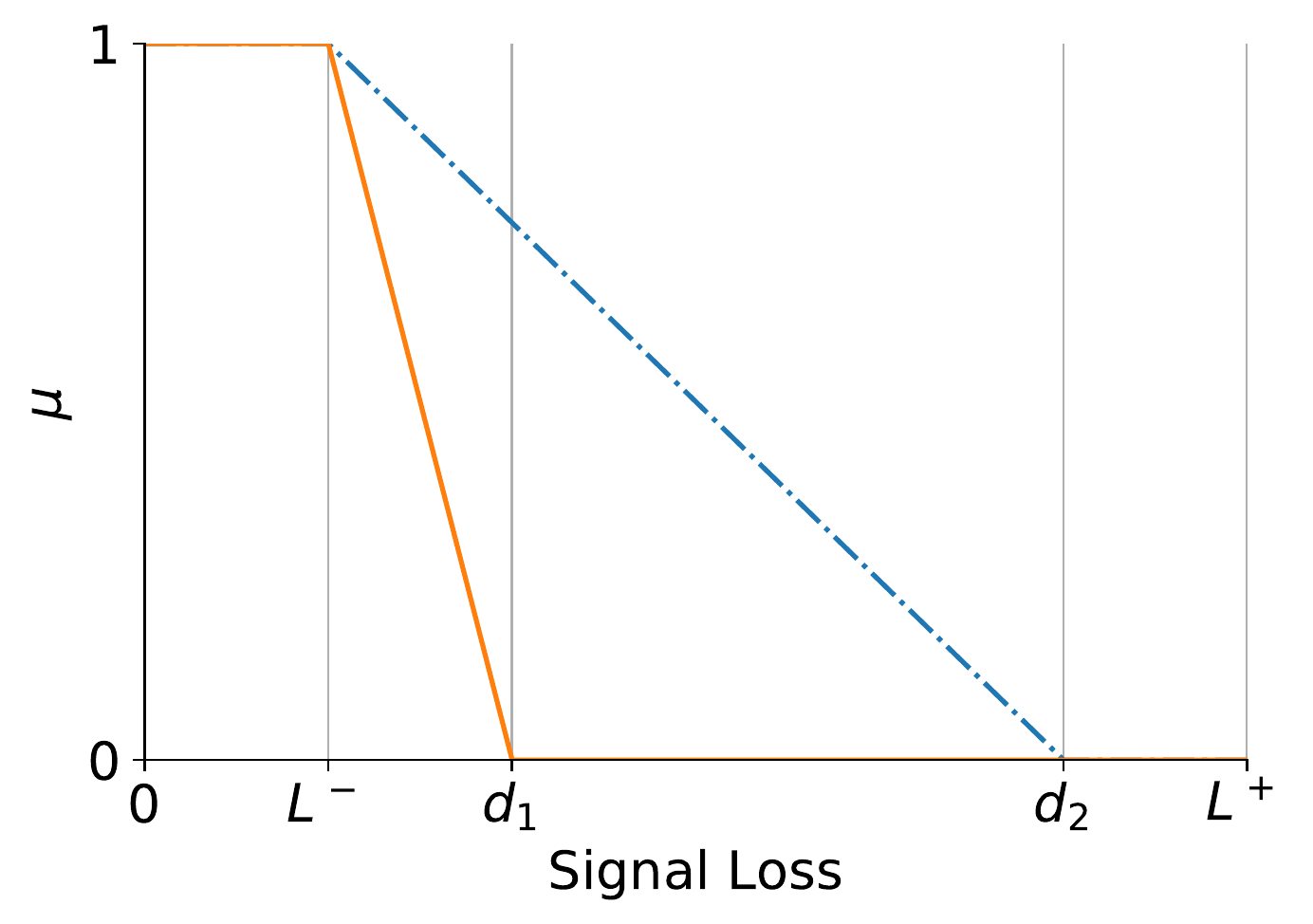}}
    \caption{Change of coverage level with respect to signal loss.}
    \label{fig2}
\end{figure}

Given the sets of user locations $\mathbf{Y_t} = (y_{1t}, \ldots , y_{nt}) \subset S^n$ and related MSTL values per time intervals $\mathbf{D_t} = (d_{1t}, \ldots , d_{nt}) \subset \mathbb{R}^n$, the discrete 3MCLP $(\text{3MCLP}_\text{D})$ can be formulated as follows:
\begin{flalign}
(\text{3MCLP}_\text{D}): &  \max\limits_{\mathbf{X}\in Q^T} 
\Omega_\text{D}(\mathbf{X}) = -p g(\mathbf{X}) + \sum\limits_{t\in\T} \sum\limits_{i\in C_t(x_t)} w_i\mu(x_t,y_{it},d_{it}).
\label{eqn:model1}
\end{flalign}

Note that the set of covered users, $C(x_t)$, can also be stated in a different way by introducing new binary variables, $z_{it}$, that denote whether or not user $i$ is covered in interval $t$. Having added these new binary variables and placing \eqref{eqn:coverage} into \eqref{eqn:model1}, we have
\begin{flalign}
    \nonumber (\text{3MCLP}_\text{D}): & \max\limits_{\mathbf{X}\in Q^T, \mathbf{Z} \in \{0,1\}^{n \times T}} \Omega_D(\mathbf{X},\mathbf{Z}) = -pg(\mathbf{X}) + \sum\limits_{i=1}^n w_i \sum\limits_{t=1}^T \frac{d_{it}-L(x_t,y_{it})}{d_{it}-L^-}z_{it} \\
    \nonumber & \phantom{.......................} = -pg(\mathbf{X}) + \sum\limits_{t=1}^T\sum\limits_{i=1}^n \left(\nu_{it}z_{it} - \kappa_{it}z_{it}L(x_t,y_{it})\right) \\
    \nonumber & \text{   s.t.} \\
    & L(x_t,y_{it})-d_{it} \leq M(1-z_{it}),\phantom{...}i\in \I, t\in \T \label{eqn:cons1} \\
    & z_{it} \in \{0,1\},\phantom{...} i \in \I, t\in\T,  \label{eqn:cons2}
\end{flalign}

\noindent where $\nu_{it}={w_id_{it}}\slash ({d_{it}-L^-})$, $\kappa_{it}=1\slash ({d_{it}-L^-})$ are auxiliary parameters, and $M$ is a sufficiently large non-negative number (e.g., $M=L^+-L^-$). Constraints~\eqref{eqn:cons1} enforce $z$ variables to be 0 when the MSLT of a user is exceeded, i.e., $d_{it}-L(x_t,y_{it})<0$. When $d_{it}-L(x_t,y_{it})\geq 0$, $z$ variables are not restricted and set to 1 since the objective is maximization. Constraints~\eqref{eqn:cons2} are binary restrictions for $z$ variables.  

Note that the objective function in $(\text{3MCLP}_\text{D})$ includes multiplication of decision variables $z$ and $L$. Therefore, we introduce new auxiliary variables, $s_{it}$, to overcome this non-linearity. As a result, the $(\text{3MCLP}_\text{D})$ can be stated as,
\begin{flalign}
    \nonumber (\text{3MCLP}_\text{D}): & \max\limits_{\mathbf{X}\in Q^T,\mathbf{Z} \in \{0,1\}^{n \times T},\mathbf{S} \in \mathbb{R}^{n \times T}} \Omega_D(\mathbf{X},\mathbf{Z},\mathbf{S})  = -pg(\mathbf{X}) + \sum\limits_{t=1}^T\sum\limits_{i=1}^n \left(\nu_{it}z_{it} - \kappa_{it}s_{it}\right) \\
    & \nonumber \text{s.t.} \\
    & \nonumber \text{Constraints~\eqref{eqn:cons1}-\eqref{eqn:cons2}} \\
    & s_{it} \leq Mz_{it},i\in\I,t\in\T \label{eqn:cons3}\\
    & s_{it} \leq L(x_t,y_{it}), i\in\I,t\in\T \label{eqn:cons4}\\
    & L(x_t,y_{it})-M(1-z_{it}) \leq s_{it} , i\in\I,t\in\T \label{eqn:cons5}\\
    & s_{it} \geq 0 , i\in\I,t\in\T. \label{eqn:cons6}
\end{flalign}

In the above formulation, when a $z$ variable is set to 0, Constraints~\eqref{eqn:cons3} and \eqref{eqn:cons6} become tight and the corresponding $s$ variable becomes 0, which yields no coverage in the objective function value. On the other hand, when a $z$ variable is set to 1, Constraints~\eqref{eqn:cons4} and \eqref{eqn:cons5} become tight and the corresponding $s$ variable becomes equal to $L$, which increases the objective function value.

\subsection{Lagrangian decomposition algorithm}
$(\text{3MCLP}_\text{D})$ can be solved by using commercial nonlinear programming solvers like BARON, but generally such an approach takes an excessively long time even for moderately sized problems. This fact motivates the development of an LDA, which is a Lagrange heuristic based on Lagrangian relaxation of nonlinear constraints, and on simplifying non-convex functions by convex approximations. LDA has been widely used within the scope of different MINLP applications such as energy storage, portfolio management, and water network designs. A detailed source of decomposition approaches in MINLP can be found in \cite{Novak2005} and the references therein. 

The LDA is shown to be efficient when a non-convex MINLP has a block-seperable structure \cite{Novak2005}, which is the case in our problem after applying a Lagrangian relaxation to $\mathrm{3MCLP}_\mathrm{D}$. Relaxing all non-convex constraints~\eqref{eqn:cons1},\eqref{eqn:cons4}, and \eqref{eqn:cons5} with non-negative multipliers $\lambda$, $\vartheta$, and $\delta$ yields the following relaxed formulation:
\begin{flalign}
    \nonumber (\text{3MCLP}_{\text{LR}}): & \max\limits_{\mathbf{X},\mathbf{Z},\mathbf{S},\mathbf{\lambda},\mathbf{\vartheta},\mathbf{\delta}} \Omega_\text{LR}(\mathbf{X},\mathbf{Z},\mathbf{S},\mathbf{\lambda},\mathbf{\vartheta},\mathbf{\delta})  = -pg(\mathbf{X}) \\
    \nonumber & \phantom{...................} +\sum\limits_{t=1}^T\sum\limits_{i=1}^n \Big(\nu_{it}z_{it} - \kappa_{it}s_{it} - \lambda_{it}[s_{it}-L(x_t,y_{it})] \\
    \nonumber & \phantom{.....................................} -\vartheta_{it}[L(x_t,y_{it})-M(1-z_{it})-s_{it}] \\
    \nonumber & \phantom{.....................................}  -\delta_{it}[L(x_t,y_{it})-d_{it}-M(1-z_{it})]\Big) \\
    & \nonumber = -pg(\mathbf{X})+\sum\limits_{t=1}^T \sum\limits_{i=1}^n \left(\omega_{1,it}z_{it}+\omega_{2,it}s_{it}+\omega_{3,it}L(x_t,y_{it})+\omega_{4,it}\right) \\
    & \nonumber \text{s.t.} \\
    & \nonumber \text{Constraints~\eqref{eqn:cons2},\eqref{eqn:cons3},\eqref{eqn:cons6}} \\
    & \lambda_{it},\vartheta_{it},\delta_{it} \geq 0,i\in\I,t\in\T, \label{eqn:conslr}
\end{flalign}

\noindent where $\mathbf{X} \in Q^T$, $\mathbf{Z} \in \{0,1\}^{n \times T}$, $\mathbf{S} \in \mathbb{R}^{n \times T}$, $\omega_{1,it}=(\nu_{it}-M(\vartheta_{it}+\delta_{it}))$, $\omega_{2,it}=(-\kappa_{it}-\lambda_{it}+\vartheta_{it})$, $\omega_{3,it}=(\lambda_{it}-\vartheta_{it}-\delta_{it})$, and $\omega_{4,it}=(M(\vartheta_{it}+\delta_{it})+\delta_{it}d_{it})$. $\text{3MCLP}_{\text{LR}}$ has (5n+3)T variables and 6nT constraints.

For given $\mathbf{\lambda}$, $\mathbf{\vartheta}$, and $\mathbf{\delta}$, $W=\sum\limits_{t=1}^T\sum\limits_{i=1}^n\omega_{4,it}$ is a constant and can be discarded while solving the problem. Moreover, $(\text{3MCLP}_\text{LR})$ can be decomposed into two sub-problems. The first sub-problem is solved for determining $\mathbf{Z}$ and $\mathbf{S}$ variables, and the second sub-problem is solved for determining $\mathbf{X}$ variables, since there is no remaining inter-dependency among these variables after relaxing the corresponding constraints, i.e.,
\begin{flalign}
    \nonumber (\text{P}_1)\ : & \max\limits_{\substack{\mathbf{Z} \in \{0,1\}^{n \times T} \\ \mathbf{S} \in \mathbb{R}^{n \times T}}}\ \Omega_{\text{P}_1}(\mathbf{Z},\mathbf{S}) = \sum\limits_{t=1}^T \sum\limits_{i=1}^n \left(\omega_{1,it}z_{it}+\omega_{2,it}s_{it}\right) \\
    \nonumber & \text{s.t.} \\
    \nonumber & \text{Constraints~\eqref{eqn:cons2},\eqref{eqn:cons3},\eqref{eqn:cons6}},\eqref{eqn:conslr} \\
    \nonumber (\text{P}_2)\ : & \max\limits_{\mathbf{X}\in Q^T}\ \Omega_{\text{P}_2}(\mathbf{X}) =-pg(\mathbf{X})+\sum\limits_{t=1}^T \sum\limits_{i=1}^n \omega_{3,it}L(x_t,y_{it}) 
\end{flalign}

$(\text{P}_1)$ can easily be solved by inspection, where $\mathbf{Z}$ and $\mathbf{S}$ variables are determined by checking $\omega_{1,it}$ and $\omega_{2,it}$ values. Since $z_{it}$ and $s_{it}$ variables are independent for each $i\in\I$ and $t\in\T$, we can determine the optimal solution by checking the multipliers of each user $i$ in each interval $t$. Recall that the maximum value of $z$ and $s$ variables are 1 and $M$ and the minimum values are 0, respectively. Moreover, $s$ variables can be positive if the associated $z$ variable is 1. Since the objective is maximization, we set $s$ and $z$ values to 0 if both $\omega_{1,it}$ and $\omega_{2,it}$ are negative. On the other hand, $z$ is set to 1 whenever $\omega_{1,it}$ is non-negative without checking $s$ value. In this case, $s$ value is set to 0 if $\omega_{2,it}$ is negative and to $M$ if $\omega_{2,it}$ is positive. When $\omega_{1,it}$ is negative and $\omega_{2,it}$ is positive, there exist two cases. If $\omega_{1,it}+M\omega_{2,it}$ value is non-negative, we set $z$ to 1 and $s$ to $M$, so that the objective can be improved, otherwise, we set both variables to 0. Table~\ref{tab:zs} summarizes the optimal solution to ($\text{P}_1$).

\begin{table}[!t]
    \centering
    \caption{Optimal solution of ($\text{P}_1$).}
    \begin{tabular}{l l r r l}
    \toprule
    \multicolumn{1}{c}{$\omega_{1,it}$} & \multicolumn{1}{c}{$\omega_{2,it}$} & \multicolumn{1}{c}{$z_{it}$} & \multicolumn{1}{c}{$s_{it}$} & \multicolumn{1}{c}{Condition} \\
    \midrule
    $\geq 0$ & $\geq 0$ & 1 & $M$ & - \\ 
    $\geq 0$ & $< 0$ & 1 & 0 & - \\ 
    $< 0$ & $\geq 0$ & 1 & $M$ & $\omega_{1,it}+M\omega_{2,it} \geq 0$ \\ 
    & & 0 & 0 & $\omega_{1,it}+M\omega_{2,it} < 0$ \\ 
    $< 0$ & $< 0$ & 0 & 0 & - \\ 
    \bottomrule
    \end{tabular}
    \label{tab:zs}
\end{table}

($\text{P}_2$) is relatively difficult to solve, since we still have a non-convex constrained optimization problem. Therefore, we further relax the exponential term in $L$ and assume that each user has LoS connection, i.e., $e^{-\beta(\theta(x,y)-\alpha)}=0, \forall x\in Q,y\in S$. Then, relaxed ($\text{P}_2$) can be stated as 
\begin{flalign}
    \nonumber (\text{P}_2^\prime)\ : & \max\limits_{\mathbf{X}\in Q^T}\ \Omega_{\text{P}_2}^\prime(\mathbf{X}) = -pg(\mathbf{X}) + \sum\limits_{t=1}^T\sum\limits_{i=1}^n \omega_{3,it}\left(F+10\eta\log_{10}\left(\norm{x_t-y_{it}}\right)+B\right)  \\
    & = O -pg(\mathbf{X})+\sum\limits_{t=1}^T\sum\limits_{i=1}^n\omega_{5,it}\log_{10}(\norm{x_t-y_{it}})
\end{flalign}

\noindent where $O=\sum\limits_{t=1}^T\sum\limits_{i=1}^n\omega_{3,it}(F+B)$, and $\omega_{5,it}=10\eta\omega_{3,it}$. $O$ is a constant term and can be discarded while solving ($\text{P}_2^\prime$). 

Note that ($\text{P}_2^\prime$) seems similar to the original problem in \eqref{eqn:model1} in terms of number of variables. However, the original problem cannot be explicitly solved since $C_t$ sets include an \textit{if} statement, which needs to be transformed to additional new constraints by introduction of new auxiliary variables, $\mathbf{Z} \in \{0,1\}^{n\times T}$. The original problem has $T(2n+3)$ variables and $6nT$ constraints. ($\text{P}_2^\prime$), on the other hand, does not have these auxiliary variables and have only $3T$ variables denoted by $\mathbf{X} \in Q^T$ with no constraints. Although this sub-problem seems smaller in size than the original problem, it is still difficult to solve the sub-problem due to non-convexity in the objective function. 

One promising approach to solve ($\text{P}_2^\prime$) would be relaxing the relocation penalty term in the objective function, dividing the sub-problem into smaller sub-problems, and independently determining $x_t \in Q$ for each interval $t\in T$. In fact, such an approach would have been used for the original formulation with reducing the problem into multiple single-interval problems by relaxing the relocation penalty, and solving each sub-problem independently, e.g. using a meta-heuristic like genetic or particle-swarm algorithms. However, these approaches would be optimal only when all parameters are the same for all intervals, which is defined as the homogeneous case in the sequel. Since such a case would not require relocation of the UAV during the planning horizon, dividing the problem with respect to time would be effective. However, in cases where parameters, $w$ and $d$, change in different intervals, which is defined as the heterogeneous case in the sequel, division approaches would suffer from relocation cost and cause sub-optimal solutions.

In ($\text{P}_2^\prime$), the relocation penalty term in the objective function is concave, since it involves negative of the sum of convex functions. However, the relaxed coverage term involves the multiplications of a concave logarithmic function with a coefficient. This coefficient is the combination of the Lagrange multipliers and it is unrestricted in sign. As we have assumed that the Lagrange multipliers are fixed, we can decompose this summation into two parts with positive and negative coefficients. Then, the summation with negative multipliers becomes a convex function, while the summation with positive multipliers becomes a concave function. As a result, we have the following relaxed sub-problem:
\begin{multline}
    (\text{P}_2^\prime)\ :\ \max\limits_{\mathbf{X}\in Q^T}\ \Omega_{\text{P}_2}^\prime(\mathbf{X}) = -pg(\mathbf{X}) \\ + \sum\limits_{t\in\T}\sum\limits_{i\in\I_t^+}\omega_{5,it}\log_{10}(\norm{x_t-y_{it}})
    +\sum\limits_{t\in\T}\sum\limits_{i\in\I_t^-}\omega_{5,it}\log_{10}(\norm{x_t-y_{it}})
\label{eqn:x}
\end{multline}

\noindent where $\I_t^+=\{i:\omega_{5,it}\geq 0\}$ and $\I_t^-=\{i:\omega_{5,it}<0\}$. Note that ($\text{P}_2^\prime$) can be solved by so-called ``Difference of Convex (DC)'' programming optimization techniques \cite{Horst1999}, which is a heuristic method and described later in the text. 

The summary of the LDA to solve ($\text{3MCLP}_\text{LDA}$) is given in Algorithm~\ref{algo:lr}. We initialize all Lagrange multipliers and run the algorithm for at most $K$ iterations. At each iteration $k\geq 0$, $\mathbf{Z}_k$, $\mathbf{S}_k$ and $\mathbf{X}_k$ variables are determined. The objective function values of the LR problem and the original problem are calculated with these values. Whenever this gap is closed, the algorithm is terminated. Otherwise, the Lagrange multipliers are updated according to the sub-gradient technique described by \cite{Fisher2004}. 

In this technique, each multiplier is updated by adding some values to its current value. This value is calculated as a multiplication of two other values. The first one is the sub-gradient of the corresponding multiplier according to the objective function of the relaxed problem that is normalized with respect to all three sub-gradient values of Lagrangean multipliers. The second variable is a non-negative real number, the step size, which should diminish throughout the algorithm. We use a common procedure from the literature to update the step size, where the step size is initialized as a real number between 0 and 2, and then halved whenever the objective function value of the relaxed problem is not improved after a fixed number of iterations. After updating the multipliers, we recursively apply the same steps to find new $\mathbf{Z}_k$, $\mathbf{S}_k$ and $\mathbf{X}_k$ values. 

As we solve a relaxed version of the second sub-problem in each iteration, the algorithm may not converge, thus, we set a maximum iteration limit $K$. The algorithm terminates after $K$ iterations with an upper bound, $\Omega_\text{UB}$, which is the minimum objective function value among all objective function values attained after solving LR problems, to the original problem unless the optimal solution is found.

\begin{algorithm}[!t]
\caption{LDA.}\label{algo:lr}
\begin{algorithmic}[1]
\renewcommand{\algorithmicrequire}{\textbf{Input:}}
\REQUIRE $Q$, $\mathbf{Y}$, $\mathbf{D}$, $\mathbf{w}$, $K$.
\renewcommand{\algorithmicrequire}{\textbf{Initialization:}}
\REQUIRE Arbitrarily initialize $\lambda_{it},\vartheta_{it},\delta_{it},\forall i\in\I,\forall t\in\T$, $k\leftarrow 0$, $\Omega_{\text{UB}}\leftarrow +\infty$.
\WHILE{$k<K$}
\STATE Find $\mathbf{Z}_k$ and $\mathbf{S}_k$ variables according to Table~\ref{tab:zs} with given $\lambda_{it},\vartheta_{it},\delta_{it}$.
\STATE Set $\mathbf{X}_k\leftarrow \arg\max\limits_{\mathbf{X}\in Q^T}\Omega_{\text{P}_2}^\prime(\mathbf{X})$ using given $\lambda_{it},\vartheta_{it},\delta_{it}$.
\IF{$\Omega_{\text{LR}}\left(\mathbf{X}_k,\mathbf{Z}_k,\mathbf{S}_k,\left\{\lambda_{it}\right\},\left\{\vartheta_{it}\right\},\left\{\delta_{it}\right\}\right)<\Omega_{\text{UB}}$}
\STATE $\Omega_{\text{UB}}\leftarrow \Omega_{\text{LR}}\left(\mathbf{X}_k,\mathbf{Z}_k,\mathbf{S}_k,\left\{\lambda_{it}\right\},\left\{\vartheta_{it}\right\},\left\{\delta_{it}\right\}\right)$
\ENDIF
\IF{$\Omega_{\text{UB}}=\Omega(\mathbf{X}_k)$}
\STATE $\textbf{break}$
\ELSE
\STATE Update $\lambda_{it},\vartheta_{it},\delta_{it}$ by using subgradient technique.
\ENDIF
\STATE $k\leftarrow k+1$.
\ENDWHILE
\STATE \textbf{return} $\Omega_{\text{UB}}$
\end{algorithmic}
\end{algorithm}

The complexity of the LDA depends on two main operations at each iteration. The first operation includes simple algebraic operations to find $\mathbf{Z}$ and $\mathbf{S}$ variables, which take $O(nT)$ time. The second operation includes finding $\mathbf{X}$ variables by applying DC programming algorithm (DCA) defined by \cite{Horst1999}. In the DCA, the objective is to create two sequences of variables, so that the first sequence of variables converges to the local optimum of the primal problem, while the second sequence of variables converges to the local optimum of the dual problem. The key point is that the symmetry between the primal and dual problems would follow a variation on the classical sub-gradient technique used in convex optimization.

The DCA starts with an arbitrary feasible solution. At each iteration of the algorithm, first a solution is found within the sub-gradient domain of the first concave function. Then, by using this solution, a new solution is found within the sub-gradient domain of the second concave function. It is shown that recursively applying these two steps converges to a local maximum after a finite number of iterations. Indeed, the DCA has two gradient calculation steps at each iteration, thus, the complexity of a single iteration is $O(T)$. Assuming that the maximum number of iterations of DCA is set to $\mathcal{K}$, the worst-case complexity of DCA is $O(\mathcal{K}T)$. Eventually, the worst-case complexity of the LDA is $O(KnT+\mathcal{K}T)$.

Note that any $\mathbf{X}_k$ found during the LDA can be used to find a lower bound for the original problem by inputting $\mathbf{X}_k$ into Equation \eqref{eqn:model1}. As a side benefit of the LDA, we can track the optimality gap of a problem while running the algorithm. In particular, we store a list that keeps all $\Omega_\text{D}(\mathbf{X}_k)$ values and report the maximum of these values as the lower bound to the original problem, $\Omega_{\text{LB}}^\ast$. In fact, the performance of LDA depends significantly on the gap between the best upper and lower bounds attained after each iteration. We illustrate how this gap changes with parameters in the computational study in Section~\ref{sec:results}.

Although our formulation and the LDA are significant improvements for a non-convex constrained optimization problem, the worst-case complexity is still exponential due to the NP-hardness of the underlying problem. Therefore, in the next section, we propose a CA approach to overcome these challenges.

\section{Continuum approximation}\label{sec:ca}
the computational burden of discrete formulation increases sharply as the numbers of users and time intervals increase. CA could be a remedy to mitigate this challenge due to less data requirements and closed or near-closed form solutions. CA approach was first proposed by \cite{Newell1971} to find optimal dispatching times of public buses and refined by \cite{Dasci2001, Ouyang2006, Cui2010, Daganzo2010, Wang2013} for several static facility location problems. Recently, \cite{Wang2017} applied the CA approach for a dynamic facility location problem, where facilities are not allowed to be closed after opening and demand is assumed to grow throughout a finite horizon. However, our approach has no assumption on user demand. Also the service area of the facilities is in 3-D space in our case. A recent survey on CA can be found in \cite{Ansari2018}. 

To the best of our knowledge, none of the CA studies incorporate the vertical dimension into the dynamic problem formulation. Moreover, the problems studied to date are related to facility location problems, where the objective is typically the minimization of the total cost of the facilities to serve all customers in a finite region. Our approach differs from the existing studies in terms of three pillars. First, we relax the static assumption on facilities and develop a CA model to a problem where a facility is allowed to move in spatio-temporal continuum without opening/closing decisions. Second, the demand (MSLT in our context) is not necessarily growing in time, but randomly determined by each user. Third, we develop a discretization procedure in the 3-D space to find the exact location sequence of a single facility instead of determining a discrete 2-D facility location plan to cover all demand.

In a CA model, all parameters are assumed to be continuous over the service area. For our problem, user weights, $\mathbf{w}\in [0,1]^n$, and demands, $\mathbf{D}\in\mathbb{R}^{n\times T}$, are approximated with continuous functions in spatio-temporal continuum, i.e. $w_{it}$ and $d_{it}$ are approximated by $w(y,t):S\times \mathbb{R} \rightarrow \mathbb{R}$ and $d(y,t):S\times \mathbb{R} \rightarrow \mathbb{R}$. All these spatial attributes are assumed to vary continuously and slowly in $y$ and $t$. To derive the solutions to the CA model, we first analyze the homogeneous case where all parameters are assumed to be constant over an infinite service area, $\mathbb{R}^2$, and infinite horizon, $\mathbb{R}$, then use these results as building blocks to extend the derivations for more general cases.

\subsection{Homogeneous case}\label{subsec:homogeneous}
In a homogeneous scenario, we assume that $\mathbf{w}$ and $\mathbf{D}$ values are constant for all $y\in \mathbb{R}^2$ and $t\in\mathbb{R}$, i.e., $w(y,t)=w$, $d(y,t)=d$. Clearly, there is no need to move the UAV-BS in this case as nothing changes for the entire time horizon. Recall that the signal loss of a user depend on two distances, horizontal and vertical. Fixing the vertical distance yields the following useful result:

\begin{proposition}\label{prop:disc}
In an infinite homogeneous plane, the optimal covering area should form a regular disc when the UAV-BS altitude is fixed.
\end{proposition}

\begin{proof}
Suppose that the UAV-BS altitude is fixed to $h$. Then, the coverage level at any point in $S$ can only change with respect to the horizontal distance, $r$. Starting from an arbitrary point in $S$, increasing $r$ values causes an increase in the overall distance between the user and the UAV-BS and a decrease in the angle. Both of these have negative impacts on signal loss. Therefore, signal loss monotonically increases with $r$. As $d$ is constant everywhere, the coverage eventually drops down to 0 after a particular $r$ value independent of the direction, since the signal loss eventually exceeds $d$. Therefore, the optimal service area should be a disc with radius $r$ within which signal loss is less than or equal to $d$ and the UAV-BS is located at the center. \QEDA
\end{proof}

With Proposition~\ref{prop:disc}, we can argue that the maximum coverage area around any location $y\in S$ can be found where the average signal loss is equal to $d$ in this area. Let $A(y,h)$ denote the size of this covering area around location $y \in S$ when the altitude of the UAV-BS is $h$. $A(y,h)$ can be considered as another approximation used to define the service area of the UAV-BS in CA, which is defined as a subset of users in the discrete formulation by $C_t \subseteq I$. Since all parameters are constant everywhere for the homogeneous case, the coverage area should be equal at every location, i.e., $A(y,h)=A(h),\ \forall y\in S$.

One key point in the CA technique is to find an approximation for the cost of service with respect to the service area so that any parameter can be expressed as a local property of point $y\in S$. For instance, \cite{Daganzo2010} proposes a total cost formula for a network in which a terminal is used to transfer products from a depot to end-users. In their proposed model, an outbound and inbound cost is aggregated to determine the total cost. The outbound cost is approximated with respect to estimated unit demand, while the inbound cost is approximated with respect to the average delivery distance within the service area.

In our context, instead of the total cost, we need to approximate the average signal loss in the covering area. As the exact loss value is not known in this area, we approximate it by assuming the loss value to increase with the same rate as the average horizontal distance increases within the covering area. \cite{Ouyang2006} show that the average distance in a disc with area $A$ can be approximated as $2\slash (3\sqrt{\pi})\times \sqrt{A}$. We use the same rate to approximate the signal loss increment within the covering area with size $A(h)$. 

Around any location $y\in S$, the signal loss is assumed to be initialized with the loss value determined as if the UAV-BS is located just atop of the point. Let $\overline{L}(h)$ denote the initial signal loss value when the UAV-BS altitude is $h$ over the point $y\in S$. While enlarging the covering area around this point, the average signal loss within this area is approximated as $\overline{L}(h)(2\slash 3)\sqrt{A(h)\slash \pi}$. Note that this approximation can be considered as a substitute of $L$ function in the discrete formulation. Instead of using exact user location, we use the service area to approximate average signal loss.

To approximate the total weight in the service area, on the other hand, we use the same approach as \cite{Daganzo2010} has used. That is, the total weight in the covering area is assumed to linearly increase with the size of the area, i.e., $w \cdot A(h)$. As a result, the CA model can be stated as follows:
\begin{flalign}
     \text{3MCLP}_\text{C}: & \max\limits_{h,A(h)\geq 0}\ \Omega_\text{C}(h,A(h))= wA(h)\left(\frac{d-\overline{L}(h)(2\slash 3)\sqrt{A(h)\slash \pi}}{d-L^-}\right),
\label{eqn:homogen}
\end{flalign}

\noindent where $\overline{L}(h)=F+10\eta\log_{10}(h)+B$. Note that for any 3MCLP problem in $\mathbb{R}^3$, this approach can be extended with different $\overline{L}(h)$ functions.

We have no discrete variable in the CA model, hence, this model is expected to be solved easier than the discrete formulation. The following is a comparison of discrete and continuous problems. In the homogeneous case, the discrete formulation can be reduced to a single-period problem without considering the relocation penalty, and subscript $t$ can be removed. In such a case, each user can be tracked with $z_i\in \{0,1\}$ and $s_i \in \mathbb{R}$ variables instead of $z_{it}$ and $s_{it}$. $z_i$ is equal to 1 if the demand of corresponding user is satisfied, and 0 otherwise, while $s_i$ is equal to signal loss value when $z_i = 1$, and 0 otherwise. Since we remove $t$ subscripts, the UAV-BS location can also be determined with a single vector, $x\in Q$. As a result, the discrete formulation would have $2n+3$ variables and $6n$ constraints. On the other hand, the CA model has only two variables, $h$ and $A(h)$. Therefore, the computational performance of the CA model would be promising.

In \eqref{eqn:homogen}, the total weight in the service area is approximated as $wA(h)$ and the average signal loss is approximated as $\overline{L}(h)(2\slash 3)\sqrt{A(h)\slash \pi}$ since the average distance within a hypothetical circular service area of size $A$ can be approximated as $(2\slash 3)\sqrt{A\slash \pi}$ \cite{Ouyang2006}. For fixed $h$, this problem can be solved by the first-order condition of $\Omega_\text{C}$, i.e., $\nabla\Omega_\text{C}(h,A)=0$, since $\Omega_\text{C}$ is concave as shown in the following lemma. 

\begin{lemma}
$\Omega_\text{C}$ is concave.
\end{lemma}

\begin{proof}
To show the concavity of $\Omega_\text{C}$, we use the Hessian matrix of the function. After a number of algebraic operations, the Hessian of $\Omega_\text{C}$ can be found as
\begin{equation}
    \nabla^{\prime\prime} \Omega_\text{C}(h,A) = 
    \begin{bmatrix}
        -\frac{20\eta wA^{3\slash 2}}{3\log(10)\sqrt{\pi}(d-L^-)h^2} & -\frac{10\eta w A^{1\slash 2}}{\log(10)\sqrt{\pi}(d-L^-)h} \\
        -\frac{10\eta w A^{1\slash 2}}{\log(10)\sqrt{\pi}(d-L^-)h} & -\frac{(F+B)wA^{-1\slash 2}}{2\sqrt{\pi}(d-L^-)}-\frac{5\eta wA^{-1\slash 2}\log_{10}(h)}{\sqrt{\pi}(d-L^-)}
    \end{bmatrix}.
\end{equation}

Note that we assume that $d\geq L^-$ and $w,\eta \geq 0$. Therefore, the first order principal minor of the above matrix is obviously negative. The second order principal minor is positive when $h\geq 10^{3\slash \log^2(10)}\approx 3.68$. As the minimum altitude allowed to hover UAVs typically starts from around 50 meters, this value is sufficiently small to assume that this principal minor is positive. Therefore, we have a negative semi-definite Hessian matrix, which yields that $\Omega_\text{C}$ is concave.\QEDA
\end{proof}

As a result, the optimal service area at altitude $h$ denoted by $A^\ast(h)$ and the corresponding total covered weight denoted by $\Omega_\text{C}(h,A^\ast(h))$ can be found as follows:
\begin{flalign}
    A^\ast(h) & = \pi\left(\frac{d}{\overline{L}(h)}\right)^2
\label{eqn:homogen2} \\
    \Omega_\text{C}(h,A^\ast(h)) & = \frac{\pi wd^3}{3(d-L^-)\overline{L}^2(h)}. \label{eqn:homogen3}
\end{flalign}

Note that \eqref{eqn:homogen3} can be maximized when $\overline{L}(h)$ takes its minimum value. Since $\overline{L}(h)$ is a monotonically increasing function in $h$, for the homogeneous case, the optimal service area can be achieved when $h$ takes its minimum value in $Q$. As a result, the optimal coverage for the homogeneous case can be achieved at any location at this minimum altitude with the optimal service area found by Equation \eqref{eqn:homogen2}.

\subsection{Heterogeneous case}\label{subsec:heterogeneous}
To solve the heterogeneous case, we first consider the case where the UAV-BS movement is ignored. Then, we propose a regularization algorithm to improve the relaxed solution by iteratively shifting the UAV-BS locations at different time intervals with repulsive forces.

When UAV-BS position is assumed to be fixed, the problem can be considered independently for each interval. Recall that $w$ and $d$ are allowed to change in this case. \cite{Wang2017} show that the parameters can be approximated with a level-based approach in the temporal continuum for dynamic problem setup without sacrificing from solution accuracy. In particular, the service area within a specific time interval is approximated by using the values at the median of the interval and then the solution found according to these values is assumed to be valid for the entire interval. 

We use a similar approach and solve all sub-problems where $w$ and $d$ values are assumed to be equal to their values at the median of an interval. Since  in our context all intervals are assumed to be identical in duration, we can determine the median of each interval beforehand. Let $\tau_t$ denote the median of interval the $t$ and $\gamma$ denote the length of each interval. Then, we can use $w(y,\tau_t)$ and $d(y,\tau_t)$ as approximations of parameters at each location $y\in S$ in interval $t$.

Since $w$ and $d$ are assumed to be slow varying functions over $S$, the service area  $A(y,h)$ should also vary slowly. Let $\varphi_t(y,A(y,h))$ denote the covered weight around $y\in S$ when the service area is $A(y,h)$ at a fixed altitude $h$ in interval $t$. Then, the following problem can be solved for each $y\in S$ for each interval $t$.
\begin{equation}
    \max\limits_{h,A(y,h)\geq 0}\ \varphi_t(y,A(y,h))=w(y,\tau_t)A(y,h)\frac{d(y,\tau_t)-\overline{L}(h)(2\slash 3)\sqrt{A(y,h)\slash \pi}}{d(y,\tau_t)-L^-}
\label{eqn:heterogen}
\end{equation}

A similar analysis to compare this model with the discrete model can be provided as we present in the homogeneous case. In a heterogeneous case, the discrete model has $T(2n+3)$ variables and $6nT$ constraints. Since all parameters are assumed to be slow varying in $S$ in the CA model, it can be solved point-by-point without a significant loss in the objective function. Suppose that $S$ is represented by $\overline{s}$ different points, which are evenly distributed with small distances from each other. Then, for each point, the optimal solution can be found by \eqref{eqn:homogen2} and \eqref{eqn:homogen3}. As a result, the optimal solution for fixed $h$ can be found with $2\overline{s}\log \overline{s}$ algebraic operations, which can computationally outperform solving an MINLP as in the discrete model.

\eqref{eqn:heterogen} can be solved with a similar approach introduced in the homogeneous case for any $y\in S$. The only difference is replacing constant parameters with the local parameters for each location $y$. Let $\varphi_t^\ast(y,A^\ast(y,h))$ denote the optimal objective function value associated with each $y$ for fixed $h$. Then, the location $y$ that yields the maximum $\varphi_t^\ast(y,A^\ast(y,h))$ can be considered as the footprint of the UAV-BS on $S$ for interval $t$, if the UAV-BS altitude is $h$. Let $y_t^\ast(h) \in S$ denote this footprint in interval $t$, then, an interval bracket search algorithm (e.g., bi-section search) can be used to find the best $h$ that maximizes $\varphi_t^\ast(y_t^\ast(h),A^\ast(y_t^\ast(h),h))$ over the domain of $h$ in $Q$. We use $h_t^\ast$ to denote this altitude for interval $t$. Then, the UAV-BS location and the total covered weight in interval $t$ are determined as $x_t^\ast = (y_t^\ast,h_t^\ast)$ and $\Omega_t^\ast \coloneqq \gamma\varphi_t(y_t^\ast,A^\ast(y_t^\ast,h_t^\ast))$, respectively. As a result the total coverage of CA denoted by $\overline{\Omega}_\textrm{C}$ can be found as $\overline{\Omega}_\textrm{C} = \sum_{t\in T} \Omega_t^\ast$.

As we relax the UAV-BS movement, the above procedure may result in a dispersed location sequence in $Q$. To alleviate such a result, we propose a regularization algorithm, where we search the neighborhood of the solution that the above procedure yields. Let $\{x_t^\ast\}$ and $\Omega^\ast=\Omega_\text{D}(\mathbf{X^\ast})$ denote the location sequence attained from the CA procedure and the true objective function value of this sequence, respectively. In the regularization algorithm, given a location sequence, at each iteration, we either select the two consecutive intervals that have the highest relocation distance with probability $\mathcal{P}$ or randomly select two intervals with probability $1-\mathcal{P}$. $\mathcal{P}$ can be interpreted as the probability of exploration in searching the neighborhood. 

After selecting the intervals, the UAV-BS location in these intervals are moved closer with a step size, $\rho$, and this movement is accepted if the objective function value is improved. Note that we use the true objective function, $\Omega_\text{D}(\cdot)$, defined in the original problem. If the objective function value improves due to this movement, then the incumbent objective function value and the location sequence is updated.

The regularization for two arbitrary intervals, $k$ and $j$, is illustrated in Fig.~\ref{fig:forces}. The transparent yellow discs depict the service areas in each interval, while the transparent blue disc shows the change in these areas. Empty red and filled green cross symbols depict the footprint of the UAV-BS before and after regularization, respectively. The angle between selected locations, $\theta_{kj}$, is used to determine horizontal and vertical repulsive forces to move the UAV-BS. 

\begin{figure}[!b]
    \centering
    \resizebox{\textwidth}{!}{\includegraphics{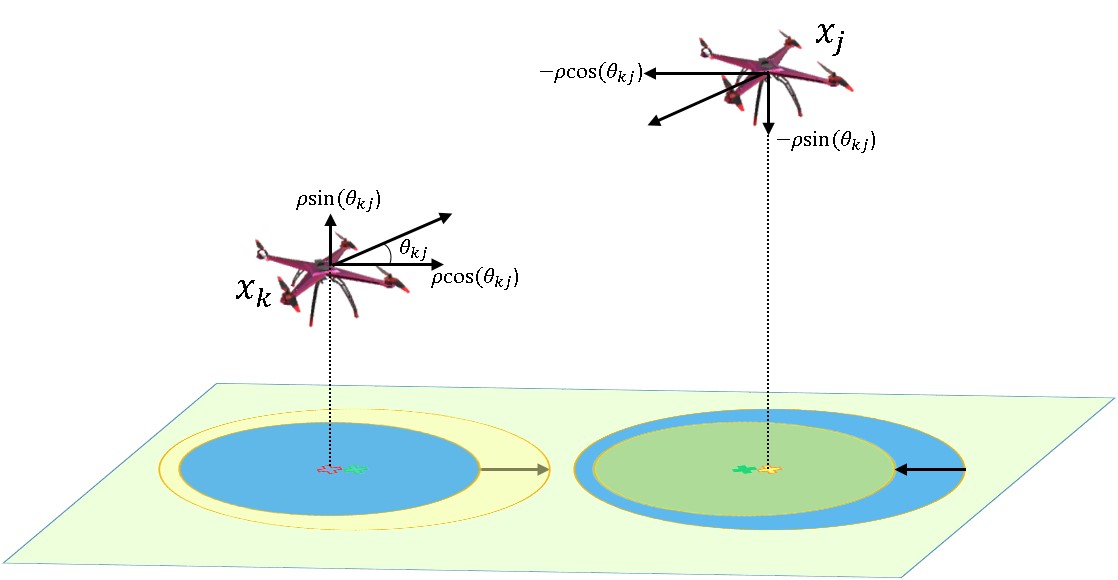}}
    \caption{Illustration of regularization.}
    \label{fig:forces}
\end{figure}

Note that lifting the altitude results in a larger service area, whereas lowering the altitude results in a smaller service area. We use such regularization based on the properties of $L(\cdot)$ to find a balance between improved angle versus worsened distance. Lifting the altitude increases the average distance to the UAV-BS. Thus, the service area is widened to boost the positive impact of increased average angle and vice versa. The rate of shrinkage or widening is applied as the rate that the altitude changes. For instance, lifting the UAV-BS from 200 m to 205 m yields widening the service area from $A$ to $(1+5\slash 200)A$, and moving down from 200 m to 195 m yields shrinking the service area from $A$ to $(1-5\slash 200)A$. Moreover, a second force with the same step size, $\rho$, is applied whenever the service area exceeds $S$ in the opposite direction of the violation.  

Note that overlapping is allowed in our problem, which is not the case in the classical 2-D CA algorithms, since our goal is to determine the locations for a single facility at different time intervals. Our algorithm terminates after $K$ consecutive iterations without improvement in the objective function value. The step size, $\rho$, is updated after each $l$ iterations to a smaller step size by multiplying with a scalar, $\varrho\in[0,1]$, to avoid aggressive movements in later stages of the algorithm. The CA algorithm is summarized in Algorithm~\ref{algo:regularize}.

\begin{algorithm}[!t]
\caption{CA Algorithm.}\label{algo:regularize}
\begin{algorithmic}[1]
\renewcommand{\algorithmicrequire}{\textbf{Input:}}
\REQUIRE $K$, $l$, $\mathcal{P}$, $\rho$, $\varrho$.
\renewcommand{\algorithmicrequire}{\textbf{Initialization:}}

\REQUIRE Determine $\{x_t\}$ by solving \eqref{eqn:heterogen}, and set $\{x_t^\ast\}$ to these locations. Set $\Omega^\ast$ as the objective function value with respect to $\{x_t^\ast\}$.
\STATE $k\leftarrow 0$, $i\leftarrow 0$, $\rho_k \leftarrow \rho$.
\WHILE{$k<K$}
\IF{$i==l$}
\STATE $\rho_k \leftarrow \rho_k \times \varrho$, $i\leftarrow 0$
\ENDIF
\STATE Select two intervals $t_1$ and $t_2$ and move $x_{t_1}$ and $x_{t_2}$ by $\rho_k$ closer to each other. Set $\Omega$ to the objective function value of this change.
\IF{$\Omega>\Omega^\ast$}
\STATE Set $\Omega^\ast \leftarrow \Omega$, and $k\leftarrow 0$. Update $\{x_t^\ast\}$ by replacing $x_{t_1}$ and $x_{t_2}$ with their new locations.
\ELSE
\STATE $k \leftarrow k+1$
\ENDIF
\STATE $i\leftarrow i+1$
\ENDWHILE\label{algo:regularize:step13}
\STATE \textbf{return} $\Omega^\ast$, $\{x_t^\ast\}$\label{algo:regularize:step14}
\end{algorithmic}
\end{algorithm}

\section{Computational results}\label{sec:results}
In this section, we provide an extensive computational study to compare the LDA and CA algorithm performances with the optimal solutions and provide insights on how heterogeneity in parameters affects the CA performance. Since there is no publicly available data set for 3MCLP, we first generate a synthetic data set and then present the comparison. All the codes are available online on \url{https://github.com/cihantugrulcicek/3DMCLP.git}.

\subsection{Data generation}
The performance of the proposed algorithms are tested through a series of 3MCLP instances with synthetic data. All instances are generated within $Q=[(0,1500)\times (0,1500)\times (50,500)]$ and $S=[(0,1500)\times (0,1500)]$. The parameters related to the signal loss function are adopted from \cite{Al-Hourani2014} for sub-urban environment which have $\eta=2$, $\alpha=4.88$, $\beta=0.43$, $\phi_\text{LoS}=0.1$, $\phi_\text{NLoS}=21$, and $f=2e9$.

For the CA model, the weight and the MSLT density functions are defined as $w(y,t) = \overline{w}[1+ \Delta_w\cos\left(\pi\norm{y}\right)] \upsilon(t)$ and $d(y,t) = \overline{d}[1+\Delta_d \cos(\pi\norm{y})] \upsilon(t)$, respectively. Here, $\overline{w}$ and $\overline{d}$ control the average values of weight and MSLT density, respectively, while $\Delta_w$ and $\Delta_d$ control the spatial variability. $\upsilon(t)$ is a function of $t\in[0,T]$ that controls the temporal variability. We consider three different functions to see how different behaviour of users in time affect the performance. The first function assumes that the parameter values increase in time, i.e., $\upsilon(t)=1+\log(1+\Delta_t t)$, while the second function assumes the values decrease in time, i.e., $\upsilon=\exp(-\Delta_t t)$, where $\Delta_t$ controls the degree of trend in time. The third function assumes that the parameters follow a random behaviour with a step function, $\upsilon(t)=1+\Delta_t\psi_k t \text{ for } t\in [\tau_{k-1},\tau_k)$, where $\psi_k$ is a uniform random variable taking values in [-1,1], which is regenerated for each interval $k$. Note that the system becomes homogeneous when $\Delta_w=\Delta_d=\Delta_t=0$. The relocation penalty value is determined with respect to average weight per area per time in an instance, i.e., $p=(1+\Delta_p)\frac{1}{T|S|}\int_T\int_S w(y,t)dydt$, where $|S|$ denotes the overall area and $\Delta_p$ controls the variation from average weight. 

For the discrete problems, we adopt a similar approach proposed in \cite{Wang2017}. $S$ is divided into $s$ spatial grid cells and a uniform random point in each cell is used to represent a user location. $\mathcal{T}$ is divided into 10 equal intervals. The weight of each discrete user location is aggregated within each cell for each interval at this random location, while the MSLT value is determined as the average MSLT value inside the cell. Since the original problem is non-convex, we use a non-convex commercial solver, BARON, available in NEOS Server \cite{Czyzyk1998} to solve discrete problems. The maximum CPU time is set to 4 hours and all other solver parameters remain as default. We also set the same time limit for the LDA, since the DC programming step of the algorithm can require an excessive solution time for some instances. 

Both the LDA and CA algorithms are coded in Python and a preliminary test has been applied for fine tuning of the parameters used in the CA algorithm. Nine different instances have been generated where $\overline{w}=0.5$, $\overline{d}=105$, and $\Delta_w=\Delta_d=\Delta_p=\Delta_t=0.2$ for all different trends of the weight and demand functions. We have drawn $l$, $\mathcal{P}$, $\rho$ and $\varrho$ values from $l\in\{50,100,150\}$, $\mathcal{P}\in\{0.1,0.3,0.5,0.7,0.9\}$, $\rho\in\{10,20,30\}$ and $\varrho\in\{0.1,0.2,0.3,0.4,0.5,0.6,0.7,0.8,0.9\}$, and solved all instances by all possible parameter combinations (405 replications in total). 

We have observed that $l$, $\rho$, and $\varrho$ are dominant parameters in terms of CPU times as these parameters are dominant in determining the step size to relocate the UAV-BS. The replications where $l$ and $\rho$ take the minimum and $\varrho$ takes the maximum among their alternative values have the longest CPU times and vice versa. However, the improvement in the objective function value is not as precise as it does in CPU time. The average improvement in the objective function value is 0.91\%, while the CPU time is almost four times longer than the average CPU time of all replications. This result is based on the fact that although decreasing the step size with smaller rates may prevent aggressive relocation of the UAV-BS, the algorithm eventually converges to similar solutions. On the other hand, higher $\mathcal{P}$ values yield a 0.84\% improvement on average in the objective function values with no additional CPU time. Therefore, we opt to set $l$, $\rho$ and $\varrho$ parameters to their average values and $\mathcal{P}$ to the highest value, i.e. $l=100$, $\rho=20$, $\varrho=0.5$, $\mathcal{P}=0.9$, for reasonable CPU time and for the sake of better precision.

\subsection{Simulation results}
The LDA and CA algorithms are run with an Intel i-5@3.20 GHz processor and 8 GB RAM under Windows 10 operating system. Both the LDA and CA algorithms are set to run for either 1000 iterations or 4 hours at most. Instances are generated according to different temporal behavior of users, where only a single function differs for each parameter at a time. As a result, we analyze 9 different scenarios where each parameter takes increasing, decreasing, and random values in time. 10 replications are generated for each scenario with $\overline{w}=0.5$, $\overline{d}=105$, and $\Delta_w=\Delta_d=\Delta_p=\Delta_t=0.2$. The optimality gap and CPU times for BARON, LDA, and CA algorithms are presented in Table~\ref{tab:res1} for different $s$ values. 

Since the CA algorithm includes random parameters, all scenarios are replicated 10 times and the average and maximum results obtained from all replications are presented. We use the maximum of the best objective function values found by BARON and the maximum of lower bounds found throughout the LDA as the final lower bound, while the minimum of upper bound values found by BARON and LDA as the final upper bound while calculating the optimality gaps of the instances.

{\renewcommand{\arraystretch}{0.6}
\begin{table}[!t]
\centering
\caption{Simulation results.}
\small
\begin{tabular}{l l r r r r r r r r r r r r r r r r r}
\toprule
&& \multicolumn{17}{c}{$d$} \\ \cline{3-19}
&& \multicolumn{5}{c}{Increase} & & \multicolumn{5}{c}{Decrease} & & \multicolumn{5}{c}{Random} \\ \cline{3-7} \cline{9-13} \cline{15-19}
&& \multicolumn{2}{c}{GAP(\%)} & & \multicolumn{2}{c}{CPU(s)} & & \multicolumn{2}{c}{GAP(\%)} & & \multicolumn{2}{c}{CPU(s)} & & \multicolumn{2}{c}{GAP(\%)} & & \multicolumn{2}{c}{CPU(s)} \\ \cline{3-4} \cline{6-7} \cline{9-10} \cline{12-13} \cline{15-16} \cline{18-19}
\multicolumn{1}{c}{$w$} & \multicolumn{1}{c}{$s$} & \multicolumn{1}{c}{BRN} & \multicolumn{1}{c}{LDA} & & \multicolumn{1}{c}{BRN} & \multicolumn{1}{c}{LDA} & &
\multicolumn{1}{c}{BRN} & \multicolumn{1}{c}{LDA} & &
\multicolumn{1}{c}{BRN} & \multicolumn{1}{c}{LDA} & &
\multicolumn{1}{c}{BRN} & \multicolumn{1}{c}{LDA} & & \multicolumn{1}{c}{BRN} & \multicolumn{1}{c}{LDA} \\
\midrule
\multicolumn{19}{l}{Increase} \\
 & 20 & - & - &  & 4335 & 7637 &  & - & - &  & 5386 & 7695 &  & - & 1.34 &  & 5698 & TL \\
 & 50 & 5.75 & 3.59 &  & TL & TL &  & 2.33 & 3.25 &  & TL & TL &  & 4.82 & 3.65 &  & TL & TL \\
 & 100 & 4.38 & 7.17 &  & TL & TL &  & 4.55 & 4.76 &  & TL & TL &  & 5.17 & 4.29 &  & TL & TL \\
 & CA & \multicolumn{2}{r}{(7.18, 8.09)} & & \multicolumn{2}{r}{(66, 78)} & & \multicolumn{2}{r}{(5.10, 6.72)} & & \multicolumn{2}{r}{(41, 65)} & & \multicolumn{2}{r}{(8.27, 8.89)} & & \multicolumn{2}{r}{(58, 74)} \\
\multicolumn{19}{l}{Decrease} \\
 & 20 & - & - &  & 6181 & 7473 &  & - & 1.61 &  & 4913 & TL &  & - & 1.98 &  & 4483 & TL \\
 & 50 & 5.90 & 2.18 &  & TL & TL &  & 5.06 & 4.11 &  & TL & TL &  & 4.97 & 3.98 &  & TL & TL \\
 & 100 & 7.99 & 5.12 &  & TL & TL &  & 7.07 & 6.95 &  & TL & TL &  & 4.52 & 7.04 &  & TL & TL \\
 & CA & \multicolumn{2}{r}{(8.20, 9.27)} & & \multicolumn{2}{r}{ (61, 68)} & & \multicolumn{2}{r}{(7.61, 8.12)} & & \multicolumn{2}{r}{(29, 67)} & & \multicolumn{2}{r}{(9.85, 10.58)} & & \multicolumn{2}{r}{(81, 91)} \\
\multicolumn{19}{l}{Random} \\
 & 20 & - & 1.26 &  & 4207 & TL &  & - & 1.00 &  & 5566 & TL &  & 1.63 & 1.72 &  & TL & TL \\
 & 50 & 3.14 & 2.27 &  & TL & TL &  & 5.69 & 5.43 &  & TL & TL &  & 4.21 & 4.94 &  & TL & TL \\
 & 100 & 4.86 & 4.89 &  & TL & TL &  & 7.75 & 7.14 &  & TL & TL &  & 4.89 & 7.12 &  & TL & TL \\
 & CA & \multicolumn{2}{r}{(8.87, 9.56)} & & \multicolumn{2}{r}{ (35, 61)} & & \multicolumn{2}{r}{(9.18, 10.30)} & & \multicolumn{2}{r}{(51, 61)} & & \multicolumn{2}{r}{(11.15, 12.71)} & & \multicolumn{2}{r}{(60, 82)} \\
\bottomrule
\multicolumn{19}{l}{- BRN: BARON, TL: Time Limit.} \\
\multicolumn{19}{l}{- The gap and CPU values in CA rows represent the average and maximum values within 10 replications.} \\
\bottomrule
\end{tabular}
\label{tab:res1}
\end{table}
}

Table~\ref{tab:res1} shows that BARON and LDA performances are within close proximity to each other when the number of users is less, while both of them suffer from an increase in the number of users and fail to find optimal solutions. BARON finds the optimal solutions for 8 instances with 20 users, whereas the LDA finds the optimal solutions to 3 of the same 8 instances. The average optimality gap of the LDA for the remaining 5 instances is 0.90\%. The average optimality gap of BARON for the remaining 19 instances where the optimal is not found is 4.98\%, while the LDA results in an average gap of 4.71\% for the same instances. We can conclude that the performance of the LDA increases with increasing grid numbers when compared to BARON. The CA algorithm outperforms both BARON and LDA in terms of CPU times by solving almost all instances less than a minute. The average of maximum optimality gaps of 10 replications is 9.36\%, while the average optimality gap is 8.38\% for the CA algorithm. 

What is surprising in Table~\ref{tab:res1} is that the CA algorithm performance worsens for instances where the weight and the MSLT densities follow a random behavior in time. The optimality gap values of instances in which at least one parameter is randomly determined in time is 2.44\% worse than the instances where both parameters have monotone behavior. This outcome can be explained as follows: in instances with monotone trends in parameters, the UAV-BS is expected to move less compared to random instances. For example, if the MSLT at a location at the beginning is relatively smaller, it follows the same trend until the end of the horizon when the temporal variability is monotone in time. Therefore, optimal service areas are expected to be around the same locations in different time intervals. On the other hand, for instances with random behavior of parameters, there is no obvious trend as we have under the monotone scenarios. Therefore, the optimal locations are likely to vary and the total movement is expected to be higher. As we relax the relocation penalty in the CA model, the CA algorithm is likely to suffer from this relaxation and results with sub-optimal solutions.

The CPU performance of the CA algorithm motivates us to use this algorithm as an initial solution generator for the exact solution procedures. To see how it may improve the solution accuracy, we resolve each instance with BARON, where the best solution found by the CA algorithm is provided as an initial solution. Table~\ref{tab:res3} shows the improvement in the objective function values after running BARON for 4 hours. Although none of the instances are solved to optimality, the CA solutions enable BARON to improve its performance by 25.78\% on the average. This finding shows that the CA algorithm can be an efficient pre-processing tool for large-scale problems.

{\renewcommand{\arraystretch}{0.8}
\begin{table}[!t]
\centering
\caption{New optimality gap values of BARON solutions after providing CA output as initial solution.}
\begin{tabular}{l l r r r r r r r r r r}
\toprule
&& \multicolumn{8}{c}{$s$}  \\ \cline{3-10}
&& \multicolumn{2}{c}{20} & & \multicolumn{2}{c}{50} & & \multicolumn{2}{c}{100} \\ \cline{3-4} \cline{6-7} \cline{9-10}
\multicolumn{1}{c}{$w$} & \multicolumn{1}{c}{$d$} & \multicolumn{1}{c}{GAP} & \multicolumn{1}{c}{Improvement} & & \multicolumn{1}{c}{GAP} & \multicolumn{1}{c}{Improvement} & & \multicolumn{1}{c}{GAP} & \multicolumn{1}{c}{Improvement} \\
\midrule
Increase & Increase & - & - &  & 4.23 & 26.35 &  & 1.89 & 56.88 \\
 & Decrease & - & - &  & 2.97 & 49.66 &  & 5.47 & 31.53 \\
 & Random & - & - &  & 2.24 & 28.75 &  & 3.78 & 22.12 \\
Decrease & Increase & - & - &  & 2.10 & 9.83 &  & 3.05 & 32.93 \\
 & Decrease & - & - &  & 3.88 & 23.34 &  & 5.68 & 19.69 \\
 & Random & - & - &  & 3.67 & 35.63 &  & 5.17 & 33.32 \\
Random & Increase & - & - &  & 2.70 & 44.00 &  & 2.20 & 57.44 \\
 & Decrease & - & - &  & 2.36 & 52.52 &  & 2.30 & 49.08 \\
 & Random & 1.21 & 25.78 &  & 2.95 & 29.87 &  & 1.98 & 59.52 \\
\bottomrule
\multicolumn{10}{l}{\textit{Note}: All values are given in percentage.} \\
\bottomrule
\end{tabular}
\label{tab:res3}
\end{table}
}

As mentioned earlier, the LDA performance highly depends on the duality gap observed through the iterations. To show how this gap improves in different instances, the gap values are depicted in Figure~\ref{fig:gapLDA}. The instance names are denoted above each plot by concatenating abbreviated trends of weight and demand parameters, respectively, e.g. inc\_inc denotes the instance in which both $w$ and $d$ increases over time.

\begin{figure}[!t]
    \centering
    \includegraphics[width=\linewidth]{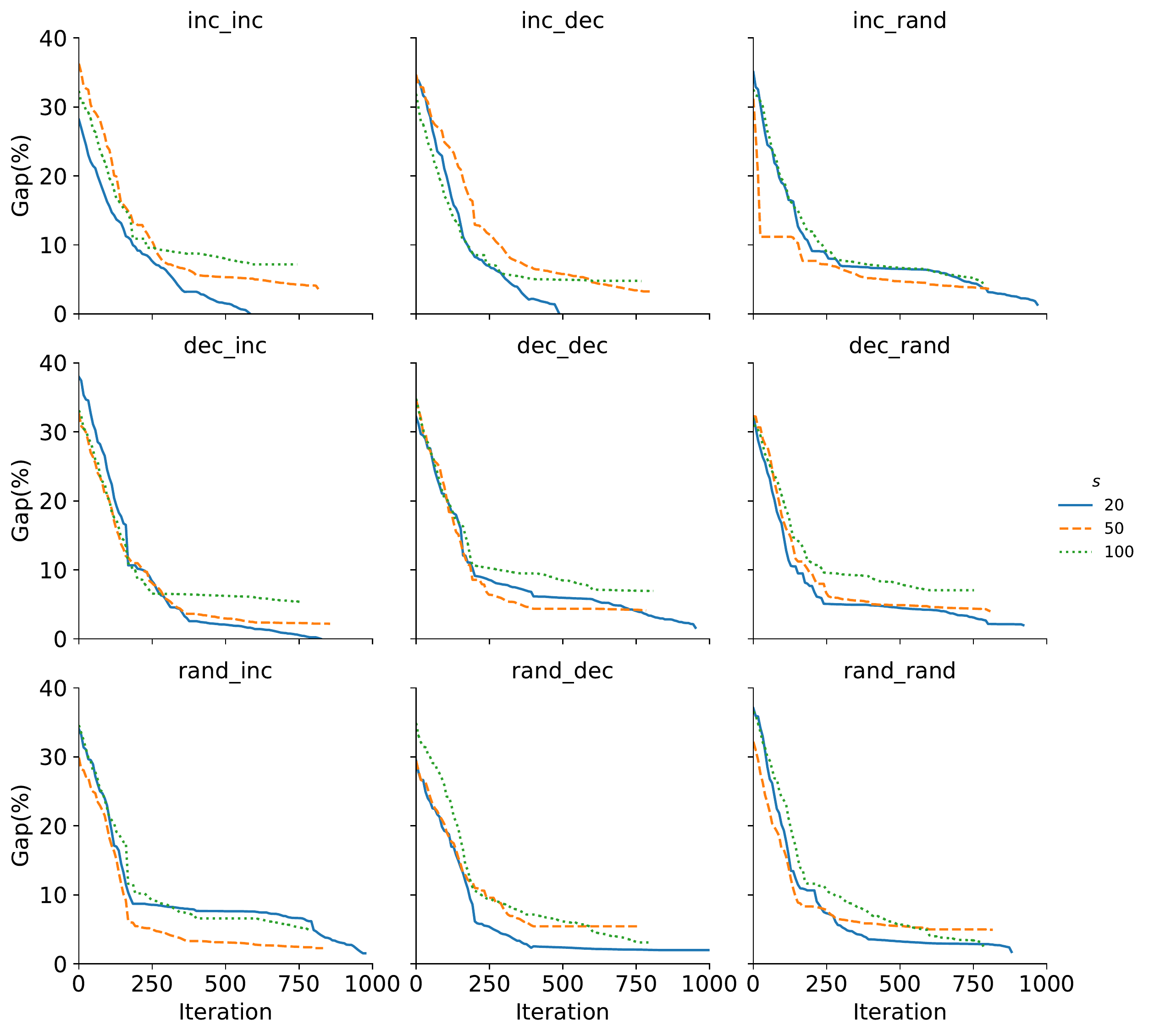}
    \caption{Change in the gap values during LDA iterations.}
    \label{fig:gapLDA}
\end{figure}

Figure~\ref{fig:gapLDA} demonstrates how smooth the gap values change through the LDA. The gap values improve after almost each iteration rather than step-wise improvements with long step sizes. As expected, the improvement rate is higher during the first iterations, and then decreases towards the termination of the algorithm. The average initial gap value of all instances is 33.18\%, and decreases to 9.34\%, 5.58\%, 4.54\%, and 3.65\% after 25, 50, 75, and 100 iterations, respectively. On average, the LDA terminates after 106, 101, and 97 iterations with the average gap values of 0.99\%, 3.71\%, and 5.09\% for the instances with 20, 50, and 100 grids, respectively. The average gap improvement rates, which are found as the rates of difference between initial and final gaps over the initial gap and averaged over all instances of corresponding grid numbers, are 95.05\%, 88.53\%, and 84.72\% for the instances with 20, 50, and 100 grids, respectively.

\begin{figure}[!b]
    \centering
    \includegraphics[width=\linewidth]{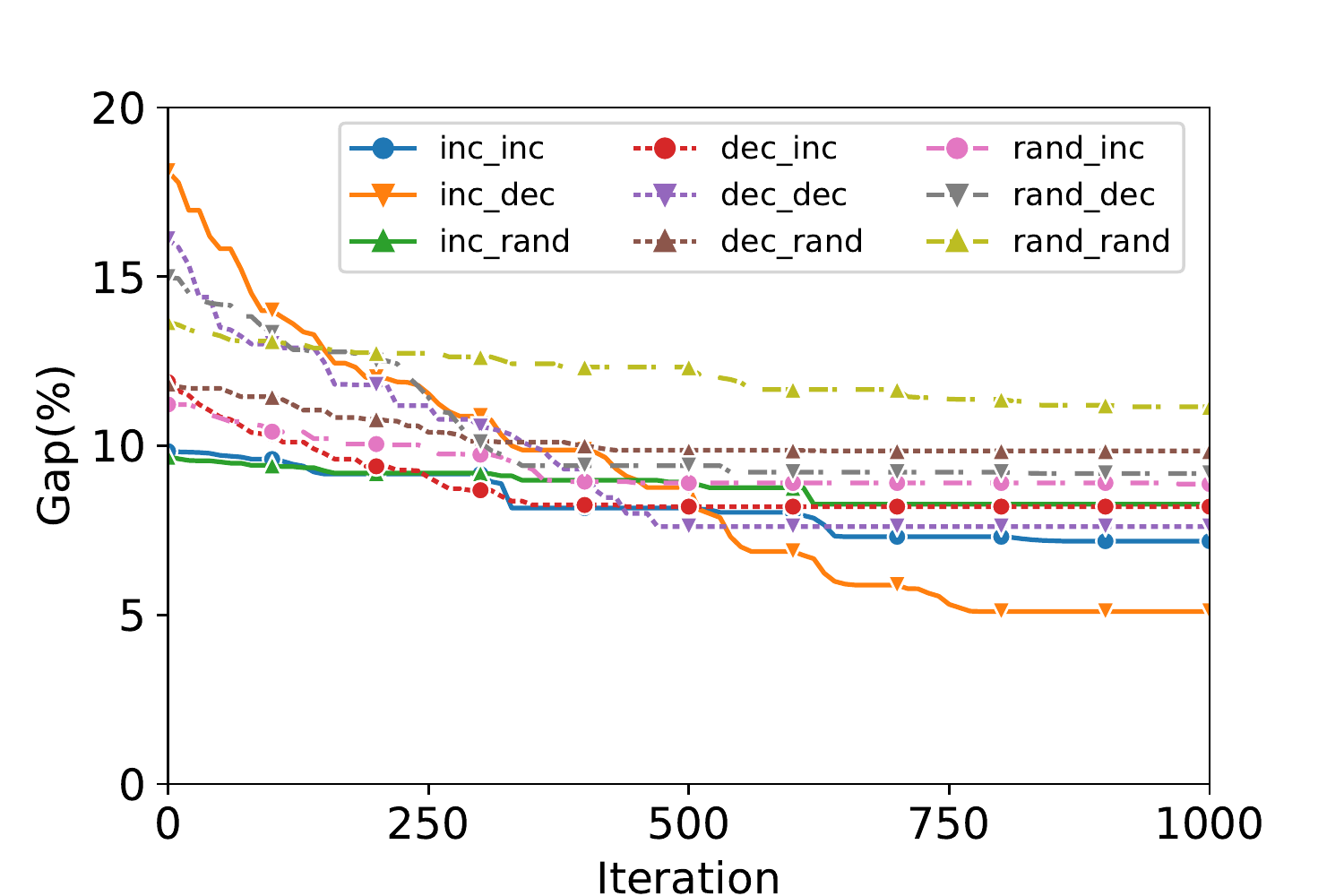}
    \caption{Change in the gap values during CA iterations.}
    \label{fig:gapCA}
\end{figure}

Figure~\ref{fig:gapCA} shows how the average optimality gap improves during the iterations of the CA algorithm. Each line in this figure illustrates the average gap values and corresponds to a specific instance as explained previously. A significant outcome from this figure is that the CA algorithm keeps improving the gap values until 755th. iteration on average. This confirms that although the convergence of CA does not seem as smooth as the LDA convergence with a higher number of step-wise decreases, smoothly decreasing step size, $\rho$, prevents the algorithm from getting stuck in sub-optimal location strategies. As a result of such improvements, the CA algorithm performs 32.44\% improvement between the initial and final gap values on average.

\begin{figure}[!b]
    \centering
    \begin{tabular}{c}
         \subfloat[]{\includegraphics[height=.35\linewidth]{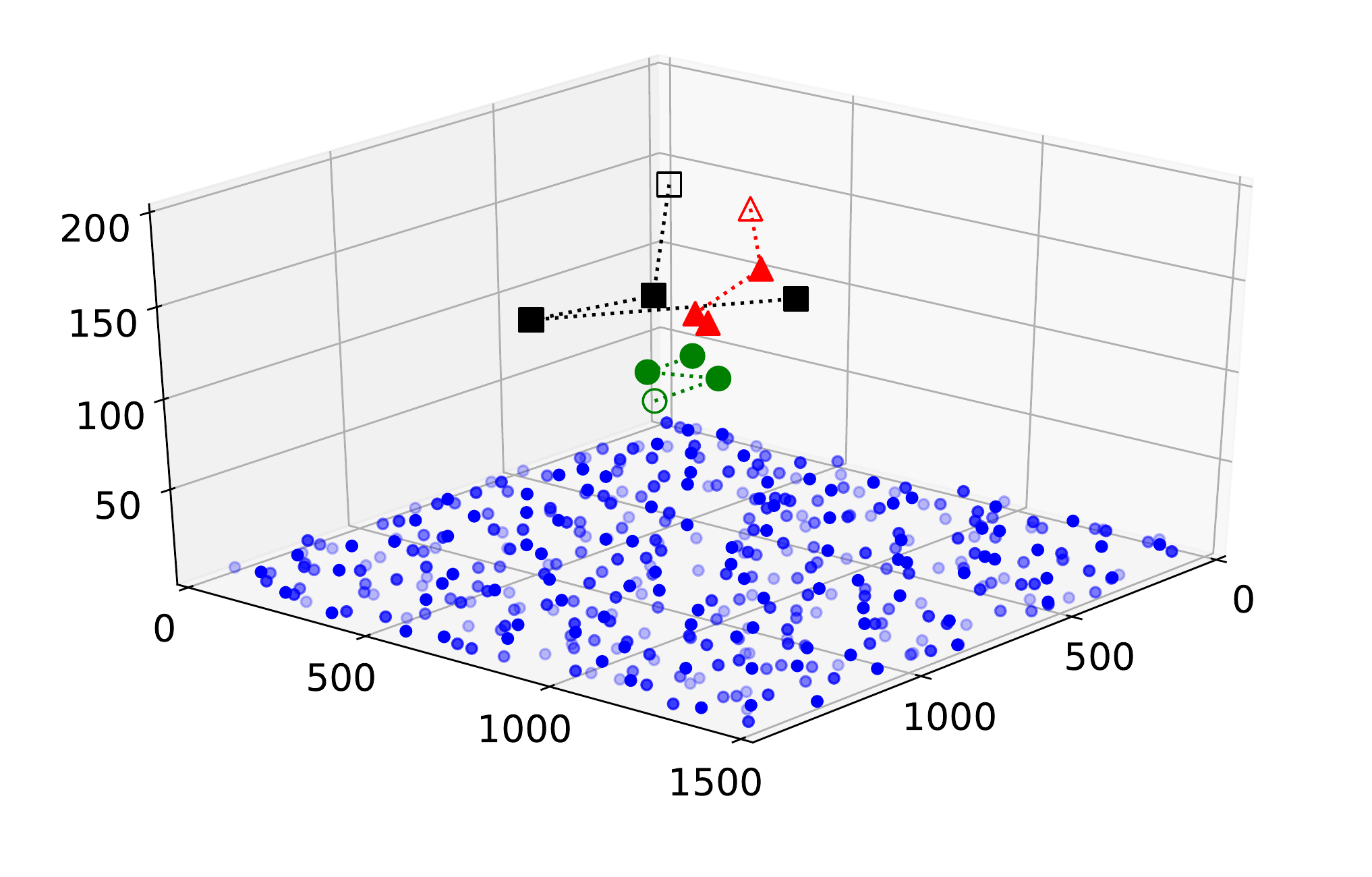}\label{fig:trajectorya}}  \\
         \subfloat[]{\includegraphics[height=.35\linewidth]{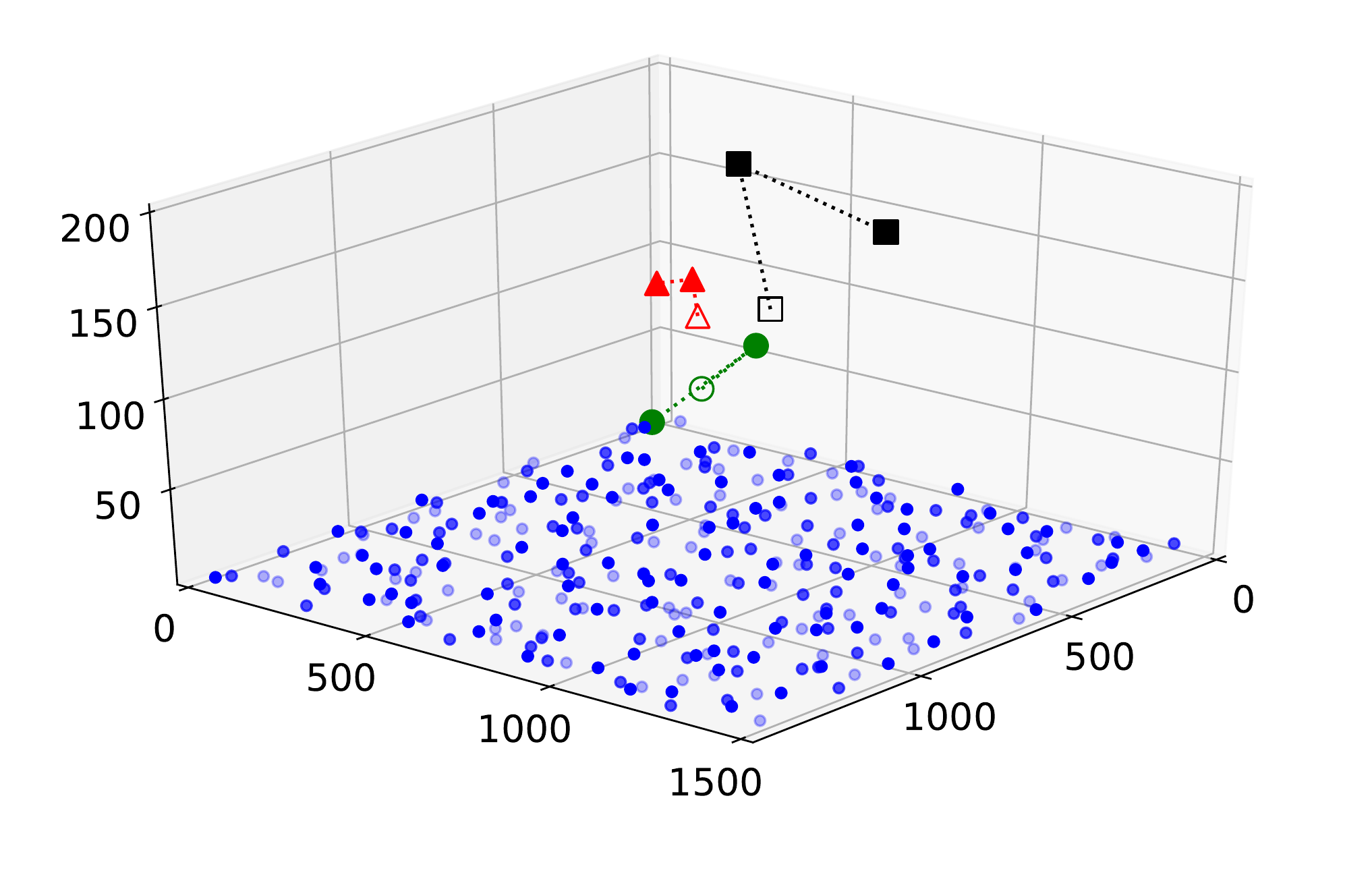}\label{fig:trajectoryb}}  \\
         \subfloat[]{\includegraphics[height=.35\linewidth]{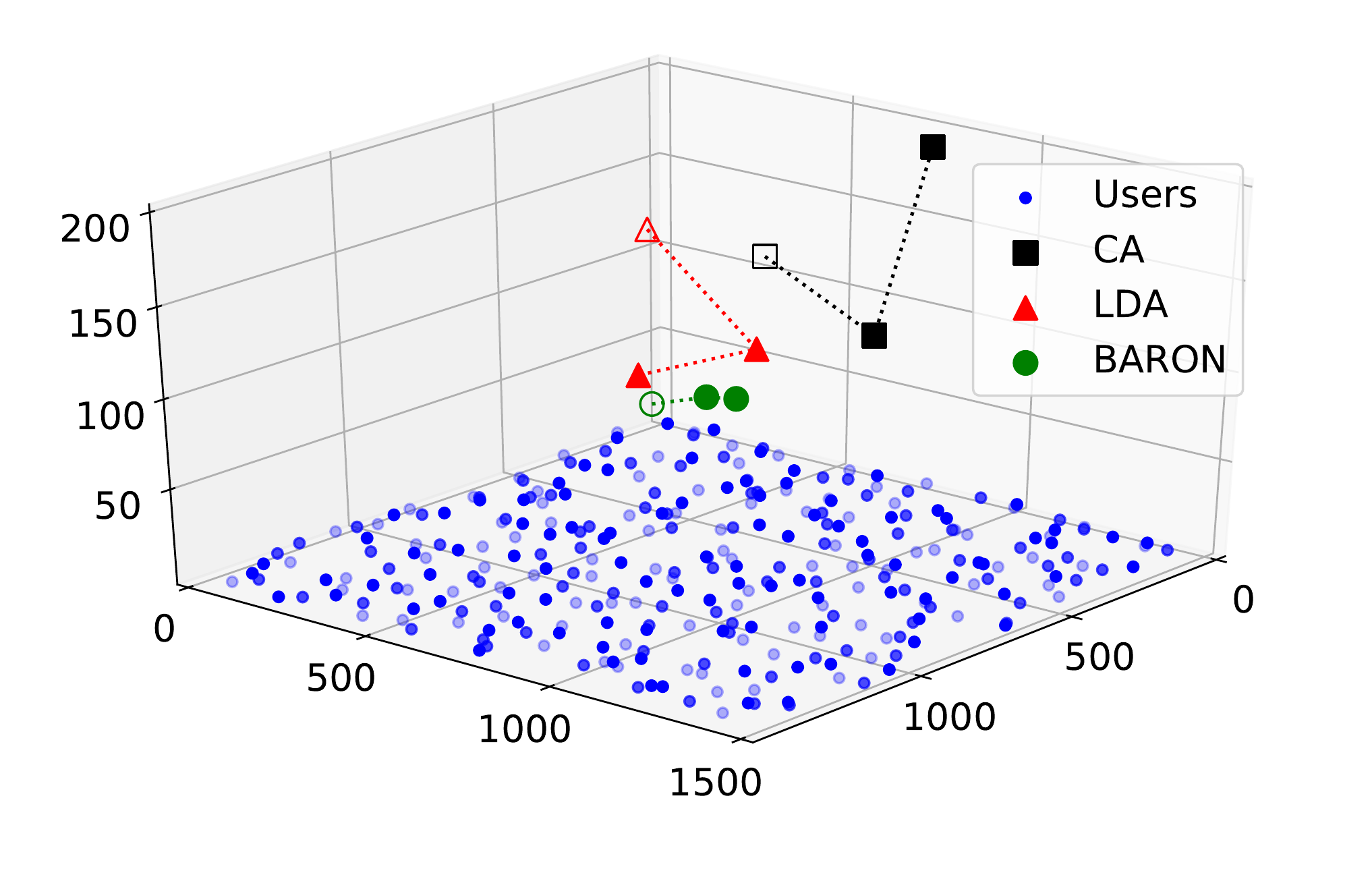}\label{fig:trajectoryc}}  
    \end{tabular}
    \caption{UAV-BS trajectories for the instance with 100 grids and decreasing weight and demand between the time intervals (a) $1-4$ (b) $5-7$ and (c) $8-10$.}
    \label{fig:trajectory}
\end{figure}

Figure~\ref{fig:trajectory} illustrates example UAV-BS trajectories found by different solution techniques. An instance with 100 grids where both $w$ and $d$ decrease is drawn as an example. The trajectories found by BARON, LDA, and CA algorithms are depicted with circle, triangle, and square markers, respectively. Figure~\ref{fig:trajectorya}, \ref{fig:trajectoryb}, and \ref{fig:trajectoryc} corresponds to the time intervals $1-4$, $5-7$, and $8-10$, respectively. Each trajectory is shown with dashed lines connecting the UAV-BS locations in succeeding intervals where empty markers denote the starting period of each sub-figure. User locations are shown with gradient colors, where the darkest markers correspond to the latest interval of the corresponding sub-figure.

Although Figure~\ref{fig:trajectory} highlights a single instance, UAV-BS locations follow similar patterns in most of the instances. BARON and LDA results seem closer while the CA results slightly differ. Since the CA algorithm attempts to improve a solution, in which the relocation is relaxed, with a greedy approach, aggressive relocation patterns are likely to be observed. On the other hand, both BARON and LDA prevent aggressive movements although the exact UAV-BS locations are different.

{
\begin{table}[!t]
\renewcommand*{\arraystretch}{0.54}
\centering
\caption{Summary of UAV-BS movement in terms of total distance and altitude changes with respect to different solution approaches.}
\begin{tabular}{
>{\small}l
>{\small}l
>{\small}r
>{\small}r
>{\small}r
>{\small}r 
>{\small}r
}
\toprule
&& \multicolumn{2}{c}{Relocation (m)} & & \multicolumn{2}{c}{Altitude Changes} \\ \cline{3-4} \cline{6-7}
{\centering $(w,d)$} & {\centering $s$} & {\centering BARON} & {\centering LDA} & & {\centering BARON} & {\centering LDA} \\
\midrule
    
inc\_inc & 20 & 2105 & 2134 & & 8.1(41.9) & 9.0(51.6)\\ 
 & 50 & 1578 & 1323 & & 8.1(42.3) & 8.9(38.3) \\ 
 & 100 & 846 & 1255 & & 8.1(28.5) & 8.9(34.5) \\ 
 & CA & \multicolumn{2}{r}{\small (2674-3000)} & & \multicolumn{2}{r}{\small (9.0-9.0)(58.2-68.1)} \\ 
inc\_dec & 20 & 2696 & 2780  & & 8.1(34.5) & 8.8(42.7)\\ 
 & 50 & 1438 & 902 & & 8.2(47.1) & 8.7(42.9)\\ 
 & 100 & 849 & 762 & & 8.1(49.5) & 8.7(47.6)\\ 
 & CA & \multicolumn{2}{r}{\small (2770-3000)} & & \multicolumn{2}{r}{\small (9.0-9.0)(60.3-69.7)}  \\ 
inc\_rand & 20 & 2232 & 2601 & & 8.1(29.9) & 8.8(32.6) \\ 
 & 50 & 1396 & 1171 & & 8.4(48.1) & 8.8(43.4) \\ 
 & 100 & 819 & 984 & & 8.4(41.2) & 9.0(45.6) \\ 
 & CA & \multicolumn{2}{r}{\small (2728-3208)} & & \multicolumn{2}{r}{\small (9.0-9.0)(78.8-90.3)}  \\ 
dec\_inc & 20 & 2168 & 2171 & & 8.1(45.5) & 8.6(47.9) \\ 
 & 50 & 1557 & 1841 & & 8.1(30.1) & 8.6(33.9) \\ 
 & 100 & 908 & 965 & & 8.0(31.8) & 8.6(33.4) \\ 
 & CA & \multicolumn{2}{r}{\small (2996-4079)} & & \multicolumn{2}{r}{\small (9.0-9.0)(61.4-72.2)}  \\ 
dec\_dec & 20 & 2442 & 2784 & & 8.2(33.5) & 8.6(37.6) \\ 
 & 50 & 1256 & 959 & & 8.2(41.9) & 8.8(37.5) \\ 
 & 100 & 995 & 701 & & 8.3(48.2) & 8.7(47.4) \\ 
 & CA & \multicolumn{2}{r}{\small (2244-3470)} & & \multicolumn{2}{r}{\small (9.0-9.0)(58.8-71.0)}  \\ 
dec\_rand & 20 & 2424 & 2861 & & 8.1(25.5) & 8.7(28.3) \\ 
 & 50 & 1228 & 1134 & & 8.3(45.9) & 8.8(45.1) \\ 
 & 100 & 852 & 772 & & 8.4(42.1) & 8.7(40.1) \\ 
 & CA & \multicolumn{2}{r}{\small (2786-3588)} & & \multicolumn{2}{r}{\small (9.0-9.0)(66.7-83.2)}  \\ 
rand\_inc & 20 & 2019 & 2829 & & 8.2(48.7) & 9.0(54.9) \\ 
 & 50 & 1137 & 1112 & & 8.4(36.1) & 9.0(36.6) \\ 
 & 100 & 802 & 746 & & 8.5(38.5) & 9.0(39.8) \\ 
 & CA & \multicolumn{2}{r}{\small (2816-3083)} & & \multicolumn{2}{r}{\small (9.0-9.0)(75.5-87.1)}  \\ 
rand\_dec & 20 & 2235 & 2489 & & 8.4(35.5) & 9.0(38.1) \\ 
 & 50 & 1376 & 1205 & & 8.4(53.9) & 9.0(51.8) \\ 
 & 100 & 823 & 1946 & & 8.4(28.8) & 9.0(39.3) \\ 
 & CA & \multicolumn{2}{r}{\small (2407-2473)} & & \multicolumn{2}{r}{\small (9.0-9.0)(72.4-79.6)}  \\ 
rand\_rand & 20 & 2271 & 2891 & & 8.8(44.6) & 9.0(46.3) \\ 
 & 50 & 1283 & 1301 & & 8.6(41.5) & 9.0(42.6) \\ 
 & 100 & 871 & 1353 & & 8.6(35.9) & 9.0(43.8) \\ 
 & CA & \multicolumn{2}{r}{\small (2540-3268)} & & \multicolumn{2}{r}{\small (9.0-9.0)(80.1-93.5)}  \\ 

\bottomrule

\end{tabular}
\label{tab:relocations}
\end{table}
}

To elaborate on the advantage of allowing the UAV-BS to move vertically, we conduct another numerical analysis. Table~\ref{tab:relocations} presents the summary of UAV-BS movement of each solution method for different behavior of $w$ and $d$. The first two columns of this table summarize the solutions in terms of total relocated distance, while the last two columns present the number of changes in altitude per interval, where the numbers in parenthesis denote the amount of change. Numbers in the CA rows represent the average and maximum values within 10 replications.

It can be observed that the UAV-BS adjusts its altitude regarding the change in parameters in all methods. The average numbers of such adjustments are 8.3, 8.8, and 9, while the average change in altitudes are 39.7, 41.6, and 68.1 meters for BARON, LDA, and CA, respectively. Regarding the total relocation amounts, BARON and LDA obviously have less movement compared to the CA algorithm. The average relocated distances are 1489, 1629, and 2944 for BARON, LDA, and CA, respectively. Moreover, the results confirm that the relocation amounts significantly decrease with an increasing number of grids. Such a result indeed depends on the fact that the distance between points selected in each grid is likely to become smaller with increasing grid numbers. Hence, it is highly likely to have points whose parameter values do not differ significantly from each other. Therefore, the UAV-BS tends to be stable and follows more contiguous trajectories to avoid aggressive movements.

In addition to the findings in Table~\ref{tab:relocations} and observations based on the UAV-BS trajectories found for each instance, we evaluate how important to let a UAV-BS move in the vertical axes instead of only flying at a single altitude level. For this purpose, we extend the original problem to a new variant where we enforce the altitude level to remain the same during the entire planning horizon. We keep everything else in the model the same. Hence, a single constraint is added to the original formulation as follows: $h_1=h_2=\ldots =h_T$, where $h_t$ denotes the altitude in the interval $t\in \mathcal{T}$. In this way, we impose our model to determine the best single altitude for the same settings with no additional updates on our assumptions.

Since adding this new constraint prevents applying the DC programming in the LDA, we could compare our results with only BARON. We solve the same instances with this new formulation by using BARON with 4-hours time limit, and find out that allowing the UAV-BS to move in the vertical axis would yield 10.21\%, 9.87\% and 9.77\% average improvement in the objective function value for 20, 50, and 100 grids, respectively. This result confirms that altitude adjustment of UAV-BSs in such networks would be promising.

To see how our algorithms perform in larger instances, we generate another data set with 200 and 500 grids. As BARON is unable to provide feasible solutions to instances with more than 100 grids within 4 hours time limit, only the performances of the LDA and CA algorithms are reported. Table~\ref{tab:ldawithca} shows the change in the optimality gap in these larger instances. The table presents the initial and final optimality gap values after 4 hours reported by the LDA. The first two columns under each grid size demonstrate the gap values obtained by running the natural LDA, while the last two columns demonstrate the gap values of the LDA with the CA solution is input as an initial solution. Observe that the average gap values in instances with 200 and 500 grids drop down from 6.92\% to 5.72\% and from 8.45\% to 7.43\% within the same CPU time when the CA solution is used, respectively. These results confirm that the CA algorithm also improves the LDA performance as it does in BARON results in Table~\ref{tab:res3}.

{\renewcommand{\arraystretch}{0.8}
\begin{table}[!t]
\centering
\caption{LDA performance for large instances with CA output given as initial solution.}
\begin{tabular}{rrrrrrrrrrrrrrrrr}
\toprule
&& \multicolumn{15}{c}{$s$} \\ \cline{3-17}
&& \multicolumn{7}{c}{200} & & \multicolumn{7}{c}{500} \\ \cline{3-9}\cline{11-17}
&& \multicolumn{3}{c}{Without CA} & & \multicolumn{3}{c}{With CA} & & \multicolumn{3}{c}{Without CA} & & \multicolumn{3}{c}{With CA} \\ \cline{3-5}\cline{7-9}\cline{11-13}\cline{15-17}
\multicolumn{1}{c}{$w$} & \multicolumn{1}{c}{$d$} & \multicolumn{1}{c}{Initial} & & \multicolumn{1}{c}{Final} & & \multicolumn{1}{c}{Initial} & & \multicolumn{1}{c}{Final} & & \multicolumn{1}{c}{Initial} & & \multicolumn{1}{c}{Final} & & \multicolumn{1}{c}{Initial} & & \multicolumn{1}{c}{Final} \\
\midrule
\multicolumn{1}{l}{Increase} & \multicolumn{1}{l}{Increase} & 33.47 &  & 5.98 &  & 11.19 &  & 5.51 &  & 33.40 &  & 8.32 &  & 11.21 &  & 7.57 \\
\multicolumn{1}{l}{} & \multicolumn{1}{l}{Decrease} & 34.02 &  & 6.81 &  & 11.42 &  & 5.83 &  & 35.45 &  & 8.52 &  & 8.43 &  & 7.41 \\
\multicolumn{1}{l}{} & \multicolumn{1}{l}{Random} & 35.36 &  & 7.14 &  & 10.59 &  & 5.78 &  & 34.44 &  & 9.09 &  & 10.49 &  & 7.55 \\
\multicolumn{1}{l}{Decrease} & \multicolumn{1}{l}{Increase} & 34.10 &  & 5.85 &  & 9.62 &  & 5.37 &  & 35.82 &  & 8.60 &  & 11.11 &  & 6.64 \\
\multicolumn{1}{l}{} & \multicolumn{1}{l}{Decrease} & 31.46 &  & 7.61 &  & 8.87 &  & 6.49 &  & 32.67 &  & 8.23 &  & 11.93 &  & 7.68 \\
\multicolumn{1}{l}{} & \multicolumn{1}{l}{Random} & 27.90 &  & 8.02 &  & 8.95 &  & 5.67 &  & 31.90 &  & 8.66 &  & 11.86 &  & 7.71 \\
\multicolumn{1}{l}{Random} & \multicolumn{1}{l}{Increase} & 26.25 &  & 5.65 &  & 9.19 &  & 5.44 &  & 36.78 &  & 7.88 &  & 10.09 &  & 7.10 \\
\multicolumn{1}{l}{} & \multicolumn{1}{l}{Decrease} & 33.09 &  & 6.98 &  & 10.96 &  & 5.67 &  & 37.16 &  & 7.92 &  & 11.30 &  & 7.63 \\
\multicolumn{1}{l}{} & \multicolumn{1}{l}{Random} & 28.88 &  & 8.22 &  & 9.32 &  & 5.74 &  & 32.13 &  & 8.85 &  & 9.93 &  & 7.62 \\
\midrule
\multicolumn{17}{l}{\textbf{Note: }All values are given in percentage.} \\
\bottomrule
\end{tabular}
\label{tab:ldawithca}
\end{table}
}

To see how heterogeneity affects the CA algorithm's performance, we generate additional series of instances, where $\Delta_w$, $\Delta_d$, $\Delta_t$, and $\Delta_p$ are drawn from $\{0,0.2,0.4,0.6,0.8\}$. We keep the same $\overline{w}$ and $\overline{d}$ values and change only one of the heterogeneity parameters at a time while fixing all other parameters to 0. Normalized average objective function values with respect to the objective function value of the homogeneous case where all heterogeneity parameters are set to 0 are illustrated in Fig.~\ref{fig:heterogeneity}. 

Fig.~\ref{fig:heterogeneity} shows that $\Delta_d$ and $\Delta_p$ have the highest and lowest impacts, respectively. Since the parameters are assumed to vary slowly in the CA model, frequent UAV-BS movement is not an expected policy, thus, the impact of $\Delta_p$ on objective function seems negligible. On the other hand, MSLT density is an important factor to determine the service area as shown in Eq.~\eqref{eqn:homogen2}, thus, high heterogeneity in $d$ has a significant impact on the location decision.

\begin{figure}[!t]
\centering
    \resizebox{.8\textwidth}{!}{\includegraphics{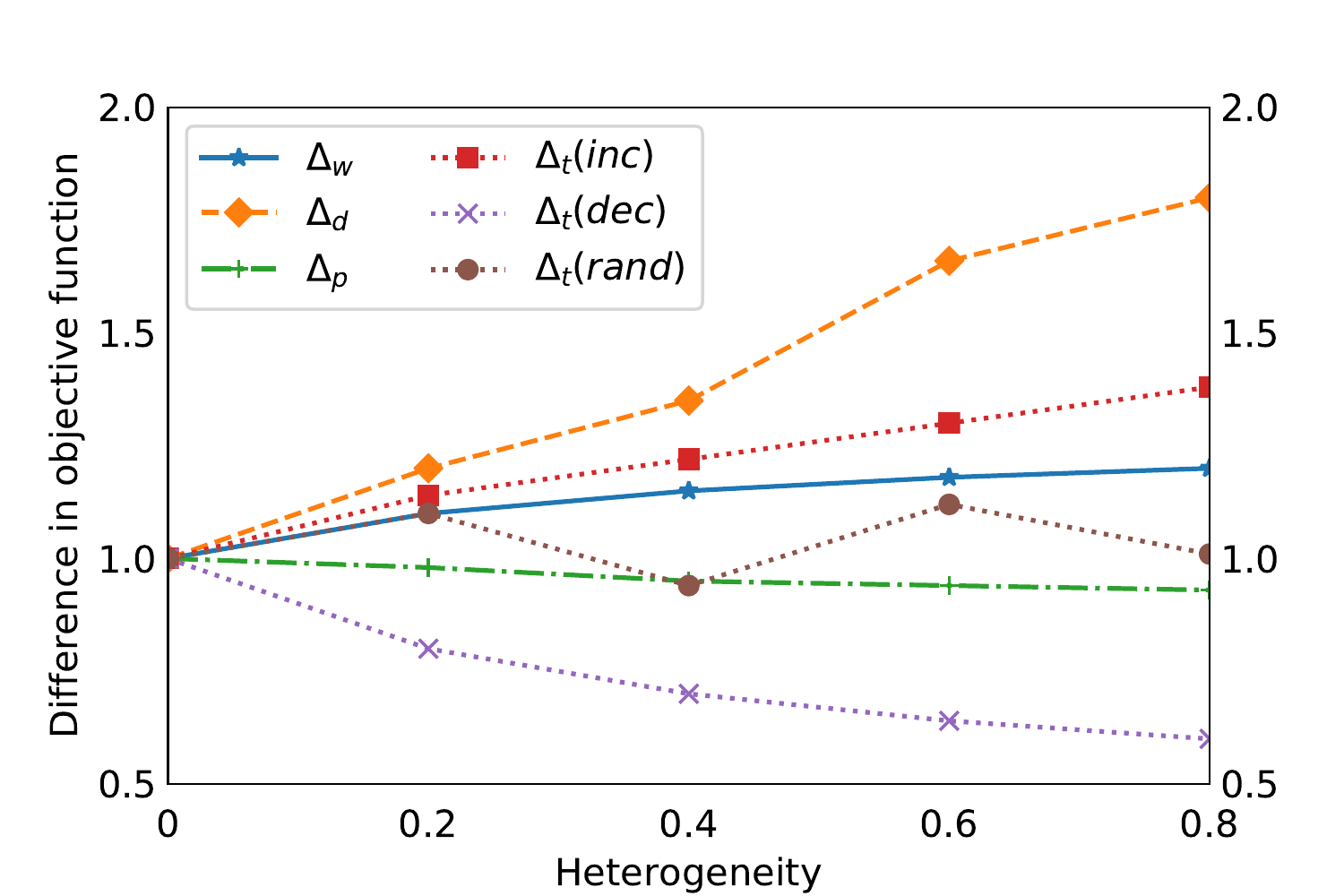}}
\caption{Impact of heterogeneity in parameters.}
\label{fig:heterogeneity}
\end{figure}

\section{Conclusions and Future Work}\label{sec:conclusion}
UAVs have been increasingly utilized in the last decade for various purposes in different industries. In the near future, their adoption and usage are expected to expand significantly and new problem versions will, most likely, need to be explored. In this study, we defined a 3-D maximal covering location problem to use UAVs as base stations in the next-generation wireless communication networks. In particular, the dynamic location of a single UAV-BS which is assumed to have infinite backhaul capacity to serve multiple ground users who are assumed to move inside a finite region within a finite time horizon is determined. We incorporated the vertical dimension into the classical 2-D maximal covering location problems and relaxed the binary coverage assumption to extend our model to more realistic cases. The problem formulation appeared to be a non-convex MINLP. Two distinct solution methods are developed to solve this problem. The custom-designed LDA efficiently determines the almost optimal solutions for medium-sized instances and finds tight bounds for larger instances. The CA algorithm, on the other hand, has the potential to be a valuable tool due to its computational performance by finding efficient solutions in reasonable times. It can also be used as an efficient heuristic to find good initial solutions for exact solution algorithms.

This study can be further extended in several directions. First, it would be interesting to address the set covering equivalent of this problem, where the objective is to determine the optimal location policies for multiple UAV-BSs given that users are covered with a minimum service quality threshold. Another possible research direction is to consider a scenario under which the service is required to satisfy certain connectivity conditions with no interruption during a specific period. This scenario brings several new constraints that tie consecutive intervals, thus, it is expected to increase complexity substantially. Another important open area is to consider disruption scenarios in which the serving time of a UAV is affected due to several factors such as weather conditions. These scenarios require a robust approach to design the network, thus, stochastic measures should be considered. A promising area would be studying learning techniques such as reinforcement learning to capture the dynamic structure of user locations and demand to decrease the prediction errors.

\end{document}